\documentclass[10pt]{amsart}

\usepackage[utf8]{inputenc}
\usepackage[english]{babel}

\usepackage{amssymb}
\usepackage{amsthm}
\usepackage{amsmath}
\usepackage{tikz-cd}
\usepackage{enumitem}
\usepackage{graphicx}
\usepackage{blkarray}
\usetikzlibrary{matrix,calc,fit}
\usepackage{comment} 
\DeclareMathOperator{\HS}{\mathrm{HS}}
\DeclareMathOperator{\HF}{\mathrm{HF}}
\DeclareMathOperator{\ini}{\mathrm{in}}

\newcommand{\truncs}[1]{s_{\leq #1}}

\usetikzlibrary{patterns}

\newtheorem{theorem}{Theorem}[section]
\newtheorem{corollary}[theorem]{Corollary}
\newtheorem{lemma}[theorem]{Lemma}

\newtheorem{proposition}[theorem]{Proposition}

\theoremstyle{definition}
\newtheorem{definition}[theorem]{Definition}
\newtheorem{example}[theorem]{Example}
\newtheorem{remark}[theorem]{Remark}
\newtheorem{setup}[theorem]{Setup}

\usepackage[bookmarksnumbered=true]{hyperref} 
\hypersetup{
     colorlinks = true,
     linkcolor = blue,
     anchorcolor = blue,
     citecolor = blue,
     filecolor = blue,
     urlcolor = blue
}

\title[The Gr\"obner basis for powers of a general linear form in a monomial CI]{The Gr\"obner basis for powers of a general linear form in a monomial complete intersection}

\author{Filip Jonsson Kling, Samuel Lundqvist, Fatemeh Mohammadi, Matthias Orth}

 \address{Filip Jonsson Kling, Department of Mathematics, Stockholm University, 106 91 Stockholm, 
 Sweden}
 \email{filip.jonsson.kling@math.su.se}

 \address{Samuel Lundqvist, Department of Mathematics, Stockholm University, 106 91 Stockholm, 
 Sweden}
 \email{samuel@math.su.se}

 \address{Fatemeh Mohammadi, Departments of Mathematics, and Computer Science, 
 KU Leuven, 
 Belgium}
\email{fatemeh.mohammadi@kuleuven.be}

 \address{Matthias Orth, Department of Mathematics, KU Leuven, BE-3001 Leuven, Belgium}
 \email{matthias.orth@kuleuven.be}

\begin{document}

\begin{abstract}
We study almost complete intersection ideals in a polynomial ring, generated by powers of all the variables together with a power of their sum. Our main result is an explicit description of the reduced Gr\"obner bases for these ideals under any term order.
Our approach is primarily combinatorial, focusing on the structure of the initial ideal. We associate a lattice path to each monomial in the vector space basis of an Artinian monomial complete intersection and introduce a reflection operation on these paths, which enables a key counting argument.
As a consequence, we provide a new proof that Artinian monomial complete intersections possess the strong Lefschetz property over fields of characteristic zero. Our results also offer new insights into the longstanding problem of classifying the weak Lefschetz property for such intersections in characteristic $p$. Furthermore, we show that the number of Gr\"obner basis elements in each degree is connected to several well-known sequences, including the (generalized) Catalan, Motzkin, and Riordan numbers, and connect these numbers 
to the study of entanglement detection in spin systems within quantum physics.
\end{abstract}
\maketitle

{\hypersetup{linkcolor=black}
{\tableofcontents}}
\section{Introduction}\label{sec:Introduction}

Monomial complete intersections, in particular ideals generated by pure powers of a subset of the variables in a polynomial ring over a field, form a well-studied class of ideals with several interesting properties. Under a suitable lexicographic ordering of the variables, these ideals are known as Clements–Lindstr\"{o}m ideals~\cite{ClementsLindstroemGeneralizationCombinatorialThmMacaulay}, and their quotient rings satisfy a generalized version of Macaulay's theorem~\cite{MacaulaySomePropertiesEnumeration}, which guarantees the existence of a lexicographic monomial ideal realizing any given admissible Hilbert function. Moreover, the Hilbert scheme over Clements–Lindstr\"{o}m rings is known to be connected~\cite{MuraiPeeva}.

In the Artinian case, when the ideal is generated by powers of all variables, the quotient ring is Gorenstein, and the finite sequence of nonzero coefficients in its Hilbert series is symmetric. Furthermore, Stanley~\cite{Stanley} and Watanabe~\cite{WATANABE1989194} independently proved that over a field
of characteristic zero such rings satisfy the strong Lefschetz property, that is, the multiplication maps induced by a general linear form $\ell$ (and its powers) are of maximal rank. In particular this means that the quotient ring 
\[\mathbf{k}[x_1,\ldots,x_n]/(x_1^{m_1},\ldots,x_n^{m_n},\ell^k) \cong \mathbf{k}[x_1,\ldots,x_n]/(x_1^{m_1},\ldots,x_n^{m_n},(x_1+\cdots+x_n)^k)\]
has the minimal possible Hilbert series among $n+1$ forms of degrees
$(m_1,\ldots,m_n,k)$, which proves the Fr\"oberg conjecture~\cite{Froberg} in the case of $n+1$ forms.

In this paper, we determine all Gr\"obner bases of the ideals
\[I_{n,\mathbf{m},k} = (x_1^{m_1},\ldots,x_n^{m_n}, (x_1+\cdots+x_n)^k). \]
While there is an algorithm to compute the Gr\"{o}bner basis of a given ideal, determining the reduced Gr\"{o}bner bases of  an infinite family of ideals is a much harder that usually requires additional techniques from combinatorics and representation theory. When a Gr\"obner basis has been established, it can also help answer questions in these areas; see e.g.~\cite{HAGLUND2018851, knutson2005annals, sturmfels1990GBs}. 

We determine the Gr\"obner basis by associating a lattice path to each monomial in the vector space basis of the quotient ring of the monomial complete intersection. We then establish a criterion to decide whether the monomial corresponding to such a path lies in the initial ideal of our ideal. This is achieved by exploiting the symmetry inherent in the monomial complete intersection and introducing a reflection map, which enables a crucial counting argument.

After describing the initial ideal, we construct the reduced Gr\"{o}bner basis of the almost complete intersection ideal. This Gr\"{o}bner basis depends only on the variable ordering induced by the given term order, and is obtained by explicitly providing, for each monomial in the initial ideal, a polynomial in the ideal whose leading term is that monomial. This construction constitutes our main result.

\begin{theorem}\label{thm:Main_Thm_Introduction}
Let $n, k\geq 1$ be integers, $\mathbf{m}=(m_1,\ldots,m_n)\in\mathbb{Z}_{\geq 2}^n$ a vector of integers, and $\mathrm{Crit}_{n,\mathbf{m},k,j}$ sets of monomials corresponding to certain lattice paths (see Definition \ref{def:Sets_Crit_n_m_k_j}).
For any monomial $s=x_1^{s_1}\cdots x_j^{s_j}=s'x_j^{s_j}\in \mathrm{Crit}_{n,\mathbf{m},k,j}$, define 
\begin{equation*}
g_s = \sum_{s''|s'} \lambda_{s''} s'' (x_j+\cdots +x_n)^{
    \deg(s)-\deg(s'') }
\end{equation*}
where 
\begin{equation*}
\lambda_{s''}=\left(\prod_{i=1}^{n-1}\frac{s_i!}{s''_i!} \binom{m_i-s''_i-1}{s_i-s''_i} \right)\cdot\frac{s_n!}{(\deg(s)-\deg(s''))!}.
\end{equation*}

Then the reduced Gr\"{o}bner basis of $I_{n,\mathbf{m},k}$ with respect to any monomial ordering $\prec$ with $x_1\succ \cdots \succ x_n$ is
\begin{equation*}
    G_{n,\mathbf{m},k}=\{x_1^{m_1},\ldots,x_n^{m_n}\}\cup \bigcup_{j=1}^n \{ g_s\,\mid\, s\in \mathrm{Crit}_{n,\mathbf{m},k,j}\}\,,
\end{equation*}
where $x_j^{m_j}$ is removed if $k+\sum_{i=1}^{j-1}(m_i-1)<m_j$.
\end{theorem}

Among the consequences of this theorem, we provide a new proof that monomial complete intersections over a field of characteristic zero have the strong Lefschetz property. Our results also offer new insight into the longstanding problem of classifying the weak Lefschetz property for monomial complete intersections in characteristic $p$. Moreover, we give a complete classification of the weak Lefschetz property in the equigenerated case, in terms of initial ideals, for $n \geq 5$.

The Gr\"{o}bner bases described in Theorem~\ref{thm:Main_Thm_Introduction} are remarkable for their enumerative properties. The integer sequences obtained by counting the elements of $G_{n,\mathbf{m},k}$ in each degree with non-trivial initial monomial, for special choices of $\mathbf{m}$ and $k$, include several well-known sequences as special cases. For instance, setting $\mathbf{m} = (2, \ldots, 2)$ yields convolutions of Catalan numbers, while $\mathbf{m} = (3, \ldots, 3)$ produces convolutions of the Riordan and Motzkin numbers. By selecting appropriate subsequences with $k=1$ and constant $\mathbf{m} = (m, \ldots, m)$, one obtains the $s$-Catalan and spin-$s$-Catalan numbers, the latter appearing in the quantum physics literature~\cite{Cohen_EtAl_Entanglement, curtright2017spin}. Both sequences have been shown to relate to Littlewood–Richardson coefficients in~\cite{Linz_s_Catalan}.

Our lattice path construction thus provides combinatorial interpretations for these and many other sequences. Conversely, we use the algebraic interpretation of these lattice paths to derive combinatorial results, such as the log-concavity of the rows of the $(m-1)$-Catalan triangle.

\medskip

\noindent \textbf{Related works.} In~\cite{booth2024weak}, the authors computed the initial ideal of the ideal generated by the squares of the variables together with the square of a linear form, and extended this analysis to the case where the generators are the respective cubes. In~\cite{kling2024grobnerbasesresolutionslefschetz}, all reduced Gr\"obner bases were determined for the ideal generated by the squares of the variables and an arbitrary power of the linear form.
The initial ideals in the square case were used in~\cite{booth2024weak} to classify when the quotient ring fails the weak Lefschetz property (WLP) due to injectivity, while~\cite{kling2024grobnerbasesresolutionslefschetz} provided a complete classification of the ideals generated by the squares of the variables and any power of the linear form with respect to the WLP.
Homological questions have also been addressed. In~\cite{Diethorn_EtAl_2025_ArXiv}, the Betti numbers of the ideal generated by the squares of the variables and the square of the linear form were determined, while~\cite{kling2024grobnerbasesresolutionslefschetz} computed the Betti numbers of the corresponding initial ideals in the squarefree algebra case, for arbitrary powers of the linear form.

\medskip
\noindent{\bf Outline.} In Section~\ref{sec:Preliminaries}, we recall basic definitions and some known results that we will make use of in our paper. In Section~\ref{sec:InitialIdeal}, we introduce our lattice path constructions and a monomial ideal induced by certain lattice paths. In Section~\ref{sec:GB}, we obtain our main result by showing that the monomial ideal defined in Section~\ref{sec:InitialIdeal} is actually the initial ideal of the ideal we study, and by obtaining the reduced Gr\"{o}bner basis. Section~\ref{sec:Applications} is devoted to special cases and applications.

\section{Background and preliminaries}\label{sec:Preliminaries}

In this section, we fix our notation and recall some key notions.
Let $\mathbf{k}$ be a field of characteristic zero, and let $R = \mathbf{k}[x_1, \ldots, x_n]$ be the polynomial ring in $n$ variables. A \emph{monomial} $s \in R$ is a power product $x_1^{s_1} \cdots x_n^{s_n}$ with $s_i \in \mathbb{Z}_{\geq 0}$. We work with the standard grading on $R$, where each variable $x_i$ has degree 1. Consequently, the \emph{degree} of a monomial $s$ is $\deg(s) = \sum_{i=1}^n s_i$. The degree of $s$ with respect to the variable $x_i$ is denoted $s_i$. 
We write $\max(s)$ for the largest index $i$ such that $x_i$ divides $s$, i.e., the largest $i$ for which $\deg_{x_i}(s) > 0$.
A \emph{monomial ideal} is an ideal generated by monomials. Every monomial ideal admits a unique minimal generating set. A polynomial $f \in R$ is called \emph{homogeneous} of degree $d$, or a \emph{form} of degree $d$, 
if it is a $\mathbf{k}$-linear combination of monomials of degree $d$. An ideal is \emph{homogeneous} if it can be generated by homogeneous polynomials.

\subsection{Gr\"{o}bner bases}
Let $\mathrm{Mon}(R)$ denote the set of all monomials in $R$. A total ordering $\prec$ of $\mathrm{Mon}(R)$ is called an \emph{(admissible) monomial ordering} if it satisfies
\begin{enumerate}
\item $\forall\; s\in \mathrm{Mon}(R)\setminus\{1\},\; 1\prec s$.
\item $\forall\; s,t,u\in \mathrm{Mon}(R),\; s\prec t\;\Longrightarrow\; u\cdot s\prec u\cdot t$.
\end{enumerate}
A well-known ordering is the \emph{reverse lexicographic ordering} $\prec$ with $x_1 \succ \cdots \succ x_n$, which is defined as follows. For any two monomials $s = x_1^{s_1} \cdots x_n^{s_n}$ and $t = x_1^{t_1} \cdots x_n^{t_n}$, we have $s \prec t$ if and only if, for the largest index $i$ where $s_i \neq t_i$, it holds that $s_i > t_i$. For example, for $n = 3$ and $x_1 \succ x_2 \succ x_3$, the monomials $s = x_1^2 x_3^3$ and $t = x_1 x_2 x_3^3$ satisfy $s \succ t$ in this ordering.

\medskip

The monomials in $R$ form a $\mathbf{k}$-basis, so each $f \in R$ has a unique expression as a $\mathbf{k}$-linear combination of monomials. The \emph{support} $\mathrm{supp}(f)$ is the set of monomials in $f$ with nonzero coefficients. Given a monomial order $\prec$, the \emph{initial monomial} $\ini_{\prec}(f)$ of a nonzero $f \in R$ is the largest monomial in $\mathrm{supp}(f)$ with respect to $\prec$. For a nonzero ideal $I \subseteq R$, the \emph{initial ideal} is $\ini_{\prec}(I) = (\ini_{\prec}(f) \mid f \in I \setminus \{0\})$. We omit the subscript $\prec$ when the order is clear.

Let $\prec$ be a monomial ordering, and let $I \subseteq R$ be a nonzero polynomial ideal. A finite set $G = \{g_1, \ldots, g_r\} \subset I \setminus \{0\}$ is called a \emph{Gr\"{o}bner basis} of $I$ if the initial ideal $\ini(I)$ is generated by the set of initial monomials $\ini(G) := \{\ini(g_1), \ldots, \ini(g_r)\}$. 
The Gr\"{o}bner basis $G$ is called \emph{minimal} if $\ini(G)$ is the minimal monomial generating set of $\ini(I)$. It is called \emph{reduced} if it is minimal and, for each $1 \leq i \leq r$, the coefficient of $\ini(g_i)$ in $g_i$ is $1$, and the set of tail monomials $\mathrm{supp}(g_i) \setminus \{\ini(g_i)\}$ is disjoint from $\ini(I)$.

The reduced Gr\"{o}bner basis of a nonzero polynomial ideal is unique for any fixed term order.

\subsection{Hilbert series}
Let $M$ be a finitely generated graded $R$-module concentrated in non-negative degrees. For each integer $d \geq 0$, let $M_d$ denote its degree-$d$ component and $\HF(M,d)=\dim_{\mathbf{k}}(M_d)$ its dimension. The \emph{Hilbert series} of $M$ is the formal power series
$\HS(M;t) = \sum_{d=0}^{\infty} \HF(M,d) \, t^d$.

In this paper, the graded algebras $A$ we consider are all Artinian, so $\HF(A,d) = 0$ for $d \gg 0$, and hence their Hilbert series are polynomials.

\begin{remark}\label{rem:Hilbert_series_via_initial_ideals_prelims}
Let $I \subseteq R$ be a nonzero homogeneous ideal, and let $\prec$ be a monomial ordering compatible with degrees, i.e., $\deg(s) < \deg(u)$ implies $s \prec u$. Then the Hilbert series of the quotient rings $R/I$ and $R/\ini_{\prec}(I)$ are equal
\[
\HS(R/I; t) = \HS(R/\ini_{\prec}(I); t).
\]
\end{remark}
\subsection{Weak and strong Lefschetz properties}
Consider an \emph{Artinian} quotient rings $A = R/I$, where $I$ is an ideal of $R$ and $A$ has finite $\mathbf{k}$-dimension. The \emph{socle degree} of $A$ is the largest integer $d \geq 0$ such that the degree $d$ component $A_d$ is nonzero. In these definitions, we use the notion of a \emph{general linear form}, i.e., a nonzero polynomial $\ell = \lambda_1 x_1 + \cdots + \lambda_n x_n \in R$ whose coefficient vector $(\lambda_1, \ldots, \lambda_n)$ lies outside a proper Zariski-closed subset of $\mathbf{k}^n$.

An Artinian algebra $A = R/I$, where $I \subset R$ is a homogeneous ideal, is said to have the \emph{weak Lefschetz property} (WLP) if, for all integers $d \geq 0$, the multiplication maps
\begin{equation}\label{eq:def_WLP}
    \mu_{\ell,d} \colon A_d \to A_{d+1}, \quad f \mapsto \ell \cdot f
\end{equation}
have maximal rank (i.e., are either injective or surjective). 
Similarly, $A$ is said to have the \emph{strong Lefschetz property} (SLP) if, for all integers $d, e \geq 0$, the multiplication maps
\begin{equation}\label{eq:def_SLP}
    \mu_{\ell,d,e} \colon A_d \to A_{d+e}, \quad f \mapsto \ell^e \cdot f
\end{equation}
have maximal rank.
\medskip

Note that if there is a linear form $\ell\in R$ such that multiplication by this form induces a map that has full rank, then a general linear form will also give a map that has full rank. Hence we may equivalently say that an algebra $A$ has the WLP (or SLP) if there is a linear form for which all the maps in \eqref{eq:def_WLP} (or in \eqref{eq:def_SLP}) have full rank. Moreover, if $A$ is defined by a monomial ideal, then it suffices to consider $\ell=x_1+\cdots + x_n$; see \cite[Theorem 2.2]{Nicklasson_SLP_two_variables}. 

We conclude this section by reviewing some properties 
of Artinian monomial complete intersections that we will make use of.

\begin{remark}\label{rem:Properties_Artinian_monomial_CIs}
Monomial ideals $P_{n,\mathbf{m}}=(x_1^{m_1},\ldots,x_n^{m_n})\subset R$, where $\mathbf{m}=(m_1,\ldots,m_n)$ is a vector of positive integers, define Artinian complete intersection algebras $A_{n,\mathbf{m}}=R/P_{n,\mathbf{m}}$ with the following properties. See \cite{book} and the references therein for details.
\begin{enumerate} 
\item $A_{n,\mathbf{m}}$ has the strong Lefschetz property.
\item The socle degree $D_{n,\mathbf{m}}$ of $A_{n,\mathbf{m}}$ is given by $D_{n,\mathbf{m}} = \sum_{i=1}^n (m_i - 1)$.
\item The Hilbert series of $A_{n,\mathbf{m}}$ is symmetric. That is, for any integer $0 \leq d \leq D_{n,\mathbf{m}}$, $\HF(A_{n,\mathbf{m}},d)=\HF(A_{n,\mathbf{m}},D_{n,\mathbf{m}} - d)$.
\item The Hilbert series of $A_{n,\mathbf{m}}$ is unimodal.
\end{enumerate}
\end{remark}

\section{Monomial ideals associated to certain lattice paths}\label{sec:InitialIdeal}
Throughout, we fix the polynomial ring $R = \mathbf{k}[x_1, \ldots, x_n]$, with positive integers $n, k$, and a tuple $\mathbf{m} = (m_1, \ldots, m_n) \in \mathbb{Z}_{\geq 2}^n$.
We consider the ideals
\[
I_{n, \mathbf{m}, k} = (x_1^{m_1}, \ldots, x_n^{m_n}, (x_1 + \cdots + x_n)^k)
\quad\text{and}\quad
P_{n, \mathbf{m}} = (x_1^{m_1}, \ldots, x_n^{m_n}).
\]
In this section, we define and study a monomial ideal $M_{n,\mathbf{m},k}$ (Definition~\ref{def:Ideal_of_critical_paths}), which we will later identify as the initial ideal of $I_{n,\mathbf{m},k}$ with respect to the reverse lexicographic term ordering. Our main result, Theorem~\ref{thm:Gen_sys_for_critical_path_ideal} provides a minimal generating set for $M_{n,\mathbf{m},k}$, described explicitly in terms of monomials associated with critical $\mathbf{m}$-admissible lattice paths. 
To prove this result, we develop a combinatorial framework based on certain lattice paths that correspond to the monomials spanning the quotient algebra $R/P_{n,\mathbf{m}}$ as a $\mathbf{k}$-vector space. Building on this structure, we show that the union of all such monomials, together with the pure powers $x_1^{m_1},\ldots,x_n^{m_n}$, yields a minimal generating set of the ideal. We also establish that this ideal is strongly $\mathbf{m}$-stable, and closely related to reverse lexicographic segment ideals.

\subsection{Lattice paths and Hilbert series}

We begin by explaining the connection between the lattice paths we will consider and the quotient rings $R/P_{n,\mathbf{m}}$.

\begin{definition}  
A lattice path is called \emph{$\mathbf{m}$-admissible} if it starts at the origin 
and is composed of $n$ steps of the types
\begin{itemize}  
\item $(a,b) \to (a+1, b+1)$ \quad (1-up),  
\item $(a,b) \to (a+1, b)$ \quad (flat),  
\item $(a,b) \to (a+1, b-e+1)$ \quad ($(e-1)$-down, for some $e$ where $m_i-1\geq e \geq 2$).  
\end{itemize}  
\end{definition} 

\begin{definition}\label{defn:MFree_Monomial}
 A monomial $t = x_1^{\alpha_1} \cdots x_n^{\alpha_n}$ in $R$ is called \emph{$\mathbf{m}$-free} if $\alpha_i < m_i$ for all $1 \leq i \leq n$.
\end{definition}

There is a bijection between $\mathbf{m}$-admissible lattice paths and $\mathbf{m}$-free monomials, given by associating each monomial with a path whose $i$-th step corresponds to the exponent of $x_i$~as~follows: 
\begin{itemize}  
\item $x_i^0$ corresponds to a 1-step up,
\item $x_i^1$ corresponds to a flat step,
\item $x_i^e$ corresponds to $(e-1)$ steps down, for $2 \leq e \leq m_i - 1$. 
\end{itemize}  

\medskip
Using this bijection, we make the following observation.

\begin{lemma}\label{lem:path_set_up_mixed_m}
Fix positive integers $n$, $d$ and a vector $\mathbf{m} \in \mathbb{Z}_{\geq 2}^n$. Then the number of $\mathbf{m}$-admissible lattice paths ending at $(n, n-d)$ equals $\HF(R/P_{n,\mathbf{m}},d)$.
\end{lemma}

\begin{proof}
Note that $\HF(R/P_{n,\mathbf{m}},d)$ counts the number of $\mathbf{m}$-free monomials in $n$ variables of degree $d$. By the bijection above, this is equal to the number of $\mathbf{m}$-admissible lattice paths consisting of exactly $n$ steps. Moreover, increasing the exponent of any variable in the monomial decreases the $y$-coordinate of the endpoint of the corresponding path by exactly one. Since the monomial $x_1 x_2 \cdots x_n$ corresponds to a path with only flat steps, it ends at $(n, 0)$. Therefore, a path corresponding to a monomial of degree $d$ must end at $(n, n - d)$, as claimed.
\end{proof}

We aim to define a monomial ideal $(M_{n,\mathbf{m},k}) \supseteq P_{n,\mathbf{m}}$, where membership of an $\mathbf{m}$-free monomial $t$ is determined by a criterion involving the geometric relationship between the lattice path of $t$ and the graph of a specific function, which we now introduce.

\begin{definition}\label{def:Red_Line}  
We define a continuous piecewise linear function $L_{n,\mathbf{m},k} : [0,n] \to \mathbb{R}$ as follows.

First, define the linear function  
\[
L_{n,\mathbf{m},k,1} : [0,1] \to \mathbb{R}, \quad x \mapsto \left(\frac{3 - m_1}{2}\right) x - \frac{k}{2}.
\]

Then, for each $i$ with $2 \leq i \leq n$, define inductively the linear function  
\[
L_{n,\mathbf{m},k,i} : [i-1, i] \to \mathbb{R}, \quad x \mapsto \left(\frac{3 - m_i}{2}\right)(x - (i-1)) + L_{n,\mathbf{m},k,i-1}(i-1).
\]

Finally, the piecewise linear function $L_{n,\mathbf{m},k}$ is defined by  
\[
L_{n,\mathbf{m},k}(x) = L_{n,\mathbf{m},k,i}(x) \quad \text{for } x \in [i-1, i], \quad 1 \leq i \leq n.
\]
\end{definition}

\begin{remark}\label{rem:Triangular_scheme}
For $1 \leq j \leq n$, let $\mathbf{m}_j = (m_1, \ldots, m_j)$.  
We can arrange the coefficients of the Hilbert series  
\[
\HS(\mathbf{k}; t), \quad \HS(\mathbf{k}[x_1]/P_{1,\mathbf{m}_1}; t), \quad \ldots, \quad \HS(\mathbf{k}[x_1,x_2,\ldots,x_{n-1}]/P_{n-1,\mathbf{m}_{n-1}}; t),\quad \HS(R / P_{n,\mathbf{m}}; t)
\]  
as columns in an approximately triangular number scheme. Here, the coefficient $1$ corresponding to $\HS(\mathbf{k}; t)$ is placed at position $(0,0)$, and a coefficient $c_{j,\ell}$ at position $(j,\ell)$ indicates that exactly $c_{j,\ell}$ many $\mathbf{m}_j$-admissible lattice paths end at the point $(j,\ell)$.
This also provides an intuitive interpretation for our definition of $L_{n,\mathbf{m},k}$. The graph of this piecewise linear function interpolates between the points located below the midpoints of the individual columns by a vertical downwards shift of $k/2$.  
See Figure~\ref{fig:Triangular_scheme_example_with_red_line} for an illustration.
\end{remark}

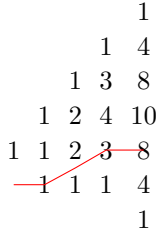
\begin{figure}[!ht]
\begin{tikzpicture}
  \matrix (M) [matrix of math nodes]
  {
     \phantom{0}& \phantom{0} & \phantom{0} & \phantom{0} & 1 \\
     \phantom{0} & \phantom{0} & \phantom{0} & 1 & 4 \\
     \phantom{0} & \phantom{0} & 1 & 3 & 8 \\
     \phantom{0} & 1 & 2 & 4 & 10 \\
    1 & 1 & 2 & 3 & 8 \\
    \phantom{0} & 1 & 1 & 1 & 4 \\
    \phantom{0} & \phantom{0} & \phantom{0} & \phantom{0} & 1 \\
  };
  
\draw[red]
                 ($ (M-6-1) $)--
                 ($ (M-6-2) $)--
                 ($ (M-5-3.south) $)--
                 ($ (M-5-4) $)--
                 ($ (M-5-5) $);
\end{tikzpicture}
\caption{Triangular array of Hilbert series coefficients together with the graph of the piecewise linear function $L_{n, \mathbf{m}, k}$ (shown in red), for parameters $n = 4$, $\mathbf{m} = (3,2,2,3)$, and $k = 2$.}
\label{fig:Triangular_scheme_example_with_red_line}
\end{figure}

When illustrating an $\mathbf{m}$-admissible lattice path, we also display the graph of $L_{n,\mathbf{m},k}$, shown as a thin red piecewise linear curve. The following provides a first example.

\begin{example}\label{ex:Geometric_configuration_Path_and_line}
Let $\mathbf{m} = (3,2,2,3)$, and consider the $\mathbf{m}$-free monomial $t = x_1 x_3 x_4^2$. Figure~\ref{fig:Geometric_configuration_Path_and_line} shows both its lattice path and the graph of $L_{4,\mathbf{m},2}$. Note that the path ends at the point $(4,0) = (4, n - \deg(t))$, as established in Lemma~\ref{lem:path_set_up_mixed_m}.
    \begin{figure}[!ht]
    \centering
    \begin{tikzpicture}
        \draw[help lines] (0,-1) grid (4,1);
        \draw[line width=0.02cm, color=red] (0,-1)--(1,-1)--(2,-0.5)--(3,0)--(4,0);
        \draw[line width=0.1cm] (0,0)--(1,0)--(2,1)--(3,1)--(4,0);
        \draw[left] (0,0) node(origin){\((0,0)\)};
        \draw[right] (4,0) node(path){\((4,0)\)};
    \end{tikzpicture}
    \caption{The graph of \(L_{4,(3,2,2,3),2}\) and the lattice path associated with the monomial \(x_1 x_3 x_4^2\).}
    \label{fig:Geometric_configuration_Path_and_line}
\end{figure}
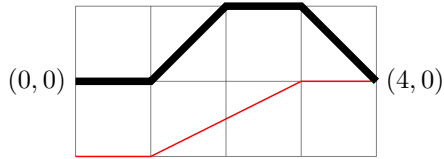
\end{example}
We now introduce a reflection operation on lattice paths associated with the $\mathbf{m}$-free monomials. We will subsequently prove that those lattice paths whose reflections remain admissible correspond bijectively to the $\mathbf{m}$-free monomials in the initial ideal of $I_{n,\mathbf{m},k}$.

\begin{definition}\label{def:reflect}
Let $(a, b)\in \mathbb{Z}^2$ be a lattice point with $0\leq a \leq n$. Reflecting this point at the intersection of the graph of $L_{n,\mathbf{m},k}$ and the vertical line $x = a$ yields a new point $(a, b')$, where
\[
b'= b - 2(b - L_{n,\mathbf{m},k}(a)).
\]
We denote this reflection by $(a,b)' := (a,b')$.
Now consider an $\mathbf{m}$-admissible lattice path $\mathbf{P}$ associated with an $\mathbf{m}$-free term of degree $d$, with vertices $P_0 = (0,0), P_1, \ldots, P_n = (n, n - d)$.
We say that the path $\mathbf{P}$ is \emph{critical} if there exists an index $0 \leq i \leq n$ such that the path  
\[
    {^i\mathbf{P}'} := (P_0, P_1, \ldots, P_{i-1}, P_i', P_{i+1}', \ldots, P_n')
\]  
is also $\mathbf{m}$-admissible.
For each critical path $\mathbf{P}$, define $\lambda(\mathbf{P})$ to be the \emph{maximal} index $i$ for which ${^i\mathbf{P}'}$ remains $\mathbf{m}$-admissible. The \emph{reflection map} associated with 
$(n,\mathbf{m},k)$ is then~defined~as  
\begin{equation*}\label{eq:Reflection_map_for_paths}
    \Lambda : \{ \mathbf{P} \mid \mathbf{P} \text{ critical and } \mathbf{m}\text{-admissible} \}
    \to \{ \mathbf{P} \mid \mathbf{P} \text{ $\mathbf{m}$-admissible} \}, \quad \mathbf{P} \mapsto {^{\lambda(\mathbf{P})}\mathbf{P}'}.
\end{equation*}
\end{definition}

\begin{example}\label{ex:Reflection}
Consider first the lattice path shown in Figure~\ref{fig:Geometric_configuration_Path_and_line}. This path touches the piecewise linear red path at its endpoint and is therefore invariant under the reflection map~$\Lambda$.

Next, consider the lattice path corresponding to the $(4,2,5)$-free monomial $x_1^3 x_3^2$, shown in black in Figure~\ref{fig:ReflectionExample}. Its image under the reflection map~$\Lambda$ is such that the initial two segments remain in black and the reflected portion appears in blue.

    \begin{figure}[!ht]
    \centering
    \begin{tikzpicture}
        \draw[help lines] (0,-4) grid (3,0);
        \draw[line width=0.02cm, color=red] (0,-2)--(1,-2.5)--(2,-2)--(3,-3);
        \draw[line width=0.1cm] (0,0)--(1,-2)--(2,-1)--(3,-2);
        \draw[line width=0.1cm, color=blue] (2,-1)--(3,-4);
        \draw[left] (0,0) node(origin){$(0,0)$};
        \draw[right] (3,-2) node(path){$(3,-2)$};
        \draw[right] (3,-4) node(path){$(3,-4)$};
    \end{tikzpicture}
    \caption{The lattice path $\mathbf{P}$ associated with the $(4,2,5)$-free monomial $x_1^3 x_3^2$ and its reflection $\Lambda(\mathbf{P})$. The reflected path corresponds to the monomial $x_1^3 x_3^4$.}
    \label{fig:ReflectionExample}
    \end{figure}
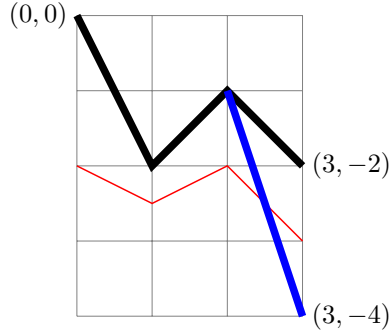
\end{example}

\begin{definition}\label{def:Ideal_of_critical_paths}
    Let $M_{n,\mathbf{m},k}$ be the set of monomials corresponding to critical $\mathbf{m}$-admissible lattice paths. We write $(M_{n,\mathbf{m},k})$ for the ideal generated by the monomials in $M_{n,\mathbf{m},k}+P_{n,\mathbf{m}}$.
\end{definition}

The reflection map~$\Lambda$ satisfies several properties, which we summarize in the following lemma.

\begin{lemma}\label{lem:Basic_reflection_arguments}
    Let $AB$ be an edge in an $\mathbf{m}$-admissible lattice path, and let $A'$, $B'$ denote the reflections of $A$ and $B$, respectively. Assume the $x$-coordinates of $A$ and $B$ are $i-1$ and $i$, respectively. Then the following statements hold:
    \begin{enumerate}
        \item[{\rm (i)}] If the edge $AB$ either touches the piecewise linear red path at $B$ or crosses it (i.e., $A$ lies above and $B$ below the line, or vice versa), then the edge $AB'$ has an admissible slope.
        \item[{\rm (ii)}]  If the slope of $AB$ is $\sigma$, then the slope of $A'B'$ is $3 - m_i - \sigma$, which is admissible.
 \item[{\rm (iii)}]  If both $AB$ and $AB'$ have admissible slopes, then $A'B$ also has an admissible slope.
 \item[{\rm (iv)}] If $AB$ is the unique edge in the path for which $AB'$ has admissible slope, then the slope of $AB$ is not $1$. 
   \end{enumerate}
\end{lemma}

\begin{proof}
(i)  If the edge $AB$ touches the red line at $B$, then $B' = B$ and hence $AB' = AB$ has admissible slope.
Now consider the case where $AB$ crosses the red line. Without loss of generality, assume $A$ lies above the red line and $B$ below. Let $(i-1, b)$ be the intersection of the red line with the vertical line $x = i-1$. Then 
$A = (i-1, b + \delta)$
for some positive integer $\delta$. Similarly, the red line intersects the vertical line $x = i$ at
$(i, b - \tfrac{m_i - 3}{2})$, and $ B = \left(i, b - \tfrac{m_i - 3}{2} - \epsilon \right)$ for some 
positive integer $\epsilon$.
Consequently,
$B' = \left(i, b - \tfrac{m_i - 3}{2} + \epsilon \right)$.
Since the slope of $AB$ is at least $2 - m_i$, we have
$\epsilon + \frac{m_i - 3}{2} + \delta \leq m_i - 2$,
or equivalently,
$\epsilon \leq \frac{m_i - 1 - 2\delta}{2}$.   
It follows that the slope of $AB'$ is bounded above by
$\epsilon - \frac{m_i - 3}{2} - \delta \leq \frac{m_i - 1 - 2\delta}{2} - \frac{m_i - 3}{2} - \delta = 1 - 2\delta\leq 1$. 
Since it is also bounded below by the admissible slope of $AB$, we conclude that $AB'$ has an admissible slope. This geometric configuration is illustrated on the left in Figure~\ref{fig:crossing_edge}.

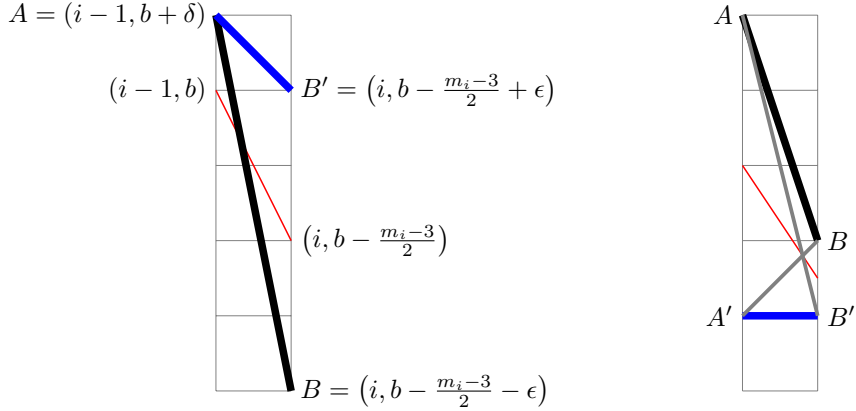
\begin{figure}[!ht]
\centering
\begin{tikzpicture}
    \draw[help lines] (0,-4) grid (1,1);
    \draw[line width=0.02cm, color=red] (0,0) -- (1,-2);
    \draw[line width=0.1cm] (0,1) -- (1,-4);
    \draw[line width=0.1cm, color=blue] (0,1) -- (1,0);
    \draw[left] (0,0) node(red){$(i-1,b)$};
    \draw[left] (0,1) node(A){$A = (i-1, b + \delta)$};
    \draw[right] (1,-4) node(B){$B = \left(i, b - \frac{m_i - 3}{2} - \epsilon \right)$};
    \draw[right] (1,0) node(B'){$B' = \left(i, b - \frac{m_i - 3}{2} + \epsilon \right)$};
    \draw[right] (1,-2) node(red_right){$\left(i, b - \frac{m_i - 3}{2} \right)$};

    \draw[help lines] (7,-4) grid (8,1);
    \draw[line width=0.02cm, color=red] (7,-1) -- (8,-2.5);
    \draw[line width=0.1cm] (7,1) -- (8,-2);
    \draw[line width=0.1cm, color=blue] (7,-3) -- (8,-3);
    \draw[line width=0.05cm, color=gray] (7,1) -- (8,-3);
    \draw[line width=0.05cm, color=gray] (7,-3) -- (8,-2);
    \draw[left] (7,1) node(A){$A$};
    \draw[left] (7,-3) node(A'){$A'$};
    \draw[right] (8,-3) node(B'){$B'$};
    \draw[right] (8,-2) node(B){$B$};
\end{tikzpicture}
\caption{(Left) Schematic view of an edge $AB$ between the vertical lines $x = i-1$ and $x = i$ in an $\mathbf{m}$-admissible lattice path crossing the red line, starting above it and ending below. The slope of $AB$ is at least $2 - m_i$. The positive integers $\delta$ and $\epsilon$ satisfy $\epsilon \leq \frac{m_i - 1 - 2\delta}{2}$. It follows that the slope of $AB'$ exceeds that of $AB$ but is at most $1 - 2\delta$, confirming admissibility of $AB'$.\\
(Right) An edge $AB$ between the vertical lines $x = i-1$ and $x = i$ in an $\mathbf{m}$-admissible lattice path with $m_i = 6$, and its reflection $A'B'$. The red segment has slope $-\frac{3}{2}$. The slopes of $AB$ and $A'B'$ sum to $3 - m_i = -3$, as do those of $AB'$ and $A'B$. Thus, if $AB'$ has an admissible slope, so does $A'B$.}
\label{fig:crossing_edge}
\end{figure}

(ii) Let $\sigma$ denote the admissible slope of $AB$, so that $\sigma \in \{1, 0, -1, \ldots, 2 - m_i\}$. Note that the slope of the red line is   
    $\frac{3 - m_i}{2}$,
which is the midpoint of the allowed slope range. It follows directly that the slope of $A'B'$ is   
    $3 - m_i - \sigma$,
and hence also admissible.

(iii)  Assume $AB'$ has admissible slope. We want to show that $A'B$ also has admissible slope. Note that by (ii), the slopes of $AB$ and $A'B'$ are $\sigma$ and $3 - m_i - \sigma$, respectively, both admissible. The slopes of $AB$ and $AB'$ differ by some integer $\delta$, and correspondingly, the slopes of $A'B'$ and $A'B$ differ by $-\delta$. Thus, the slopes of $AB'$ and $A'B$ sum to $3 - m_i$. Therefore, admissibility of $AB'$ implies admissibility of $A'B$. See the right side of Figure~\ref{fig:crossing_edge} for an illustration.

(iv) Assume, as in (i), that $(i-1, b)$ is the intersection of the red line with $x = i - 1$, and let $A = (i-1, b + \delta)$. If the edge $AB$ has slope $1$, then $B = (i, b + \delta + 1)$. Since the red line intersects $x = i$ at the point $(i, b - \frac{m_i - 3}{2})$, it follows that $B' = (i, b - \delta - m_i + 2)$. Hence, the slope of $AB'$ is $2 - m_i - 2\delta$, which is admissible if and only if $\delta = 0$. In other words, $A$ must lie on the red line. This cannot occur if $i = 1$, and if $i \geq 2$, then the edge ending at $A$ has a reflection with admissible slope by (i). Thus, there must exist another edge $CD$ such that $CD'$ has admissible slope.
\end{proof}

To understand how the symmetry of the Gorenstein algebra $R/P_{n,\mathbf{m}}$ relates the degrees of $\mathbf{m}$-free monomials via path reflections, we examine the analyze of the corresponding lattice paths.

\begin{remark}\label{rem:Mixed_socle_and_related_degrees}
Let $d' = \left(\sum_{i=1}^n m_i\right) - n - d + k$. We claim that any admissible reflection of a path associated with an $\mathbf{m}$-free monomial of degree $d$ yields a path associated with an $\mathbf{m}$-free monomial of degree $d'$. 
To see this, note that the socle degree of the algebra $R/P_{n,\mathbf{m}}$ is $\sum (m_i - 1)$. By construction, the graph of the piecewise linear function $L_{n,\mathbf{m},k}$ intersects the vertical line $x = n$ at the point
    $\textstyle{Q = \left(n, \frac{3n - \left(\sum m_i\right) - k}{2} \right)}$.
An $\mathbf{m}$-admissible lattice path corresponding to an $\mathbf{m}$-free monomial of degree $d$ ends at $Q_d = (n, n - d)$. Since $R/P_{n,\mathbf{m}}$ is Gorenstein and $(d-k)+d'$ is the socle degree of $R/P_{n,\mathbf{m}}$, we have
\begin{equation*}\label{eq:Georenstein_symmetry}
\HF(R/P_{n,\mathbf{m}},d-k) = \HF(R/P_{n,\mathbf{m}},d').
\end{equation*}
The path of degree $d'$ ends at $Q_{d'} = (n, n - d') = \left(n, 2n + d - \sum m_i - k\right)$. The average of the $y$-coordinates of $Q_d$ and $Q_{d'}$ is
    $\textstyle{\frac{n - d + \left(2n + d - \sum m_i - k\right)}{2} = \frac{3n - \sum m_i - k}{2}}$,
which matches the $y$-coordinate of $Q$. Hence, any admissible reflection of a path of degree $d$ yields a path of degree $d'$.
\end{remark}

Remark~\ref{rem:Mixed_socle_and_related_degrees} motivates the following degree-specific refinement of the reflection map $\Lambda$.

\begin{definition}\label{def:Degreewise_reflection_maps}
Let $d$ be the degree of an $\mathbf{m}$-free monomial and set $d' = \left(\sum m_i\right) - n - d + k$. We define the degree-wise reflection map $\Lambda_d$ as the restriction of $\Lambda$ to the set of critical $\mathbf{m}$-admissible paths ending at $(n, n - d)$
\begin{equation*}\label{eq:Reflection_map_for_paths_of_degree_d}
    \Lambda_d : 
    \left\{
        \begin{array}{l}
            \mathbf{P} \text{ critical and } \mathbf{m}\text{-admissible} \\
            \text{with endpoint } (n, n-d)
        \end{array}
    \right\}
    \longrightarrow
    \left\{
        \begin{array}{l}
            \mathbf{P} \text{ critical and } \mathbf{m}\text{-admissible} \\
            \text{with endpoint } (n, n-d')
        \end{array}
    \right\}, \quad
    \mathbf{P} \mapsto \Lambda(\mathbf{P}).
\end{equation*}
\end{definition}

We now show that each map $\Lambda_d$ is a bijection, with inverse $\Lambda_{d'}$. The well-definedness of $\Lambda_d$ will follow from the proof of the next lemma.

\begin{lemma}\label{lem:reflectionbijection}
    Let $d$ be the degree of an $\mathbf{m}$-free monomial and let $d' = \left( \sum m_i \right) - n - d + k$. Then the map $\Lambda_d$ is a bijection, with inverse $\Lambda_{d'}$.
\end{lemma}

\begin{proof}
    Since $(d')' = d$, we may assume without loss of generality that $d \leq d'$. Consider any admissible path $\mathbf{P}$ associated with an $\mathbf{m}$-free monomial of degree $d'$. Since this path starts above the piecewise linear red path and ends on or below it, there exists an edge $AB$ in the path which either touches the red path at $B$ or crosses it from above to below. By Lemma~\ref{lem:Basic_reflection_arguments}, the reflected edge $AB'$ is also admissible and reflecting $B$ and each subsequent vertex of the path results in an admissible path. Thus $\mathbf{P}$ is critical, that is, in the domain of the map $\Lambda_{d'}$. The reflected path $\Lambda_{d'}(\mathbf{P}) = \Lambda(\mathbf{P})$ ends at $(n, n - d)$ by Remark~\ref{rem:Mixed_socle_and_related_degrees}.

    We now show that $\Lambda_d$ is well-defined and that $\Lambda_d^{-1} = \Lambda_{d'}$. Again, consider the path $\mathbf{P}$ described above and its reflection $\Lambda_{d'}(\mathbf{P})$ ending at $(n, n - d)$. Suppose the reflection started at the vertex $B$ of the edge $AB$ of $\mathbf{P}$. Then for the path $\Lambda_{d'}(\mathbf{P})$, a reflection can be started at the vertex $B'$ of its edge $AB'$ and if this is the last possible start of a reflection, we have $\Lambda_d(\Lambda_{d'}(\mathbf{P})) = \mathbf{P}$. It remains to prove that no later start is possible. So let $C'D'$ be an edge of $\Lambda_{d'}(\mathbf{P})$ occurring after $AB'$, corresponding to the edge $CD$ in $\mathbf{P}$. If $C'D$ were admissible, that is, if a reflection could be started at $D'$ in $\Lambda_{d'}(\mathbf{P})$, then by Lemma~\ref{lem:Basic_reflection_arguments}, the edge $CD'$ would also be admissible, meaning a reflection for $\mathbf{P}$ could be started at $D$, which comes after $B$, contradicting the definition of the map $\Lambda_{d'}$.

    An analogous argument shows that $\Lambda_{d'}(\Lambda_d(\mathbf{P})) = \mathbf{P}$ for any critical path $\mathbf{P}$ ending at $(n, n - d)$. Hence $\Lambda_{d'} = \Lambda_d^{-1}$ as claimed.
\end{proof}

As a consequence of Lemma~\ref{lem:reflectionbijection}, we obtain an enumeration of critical paths ending at a given point. This result will be used to relate the initial ideal $\ini(I_{n,\mathbf{m},k})$ and the ideal $(M_{n,\mathbf{m},k})$ via an equality of Hilbert series. We will first compare the Hilbert series of $(M_{n,\mathbf{m},k})$ and $P_{n,\mathbf{m}}$.

\begin{proposition}\label{prop:Counting_Critical_Paths}
    The Hilbert series of the quotient rings $R/(M_{n,\mathbf{m},k})$ and $R/P_{n,\mathbf{m}}$ satisfy
    \begin{equation*}
        \HS\big(R/(M_{n,\mathbf{m},k}); t\big) = \big[(1 - t^k) \, \HS(R/P_{n,\mathbf{m}}; t)\big],
    \end{equation*}
    where the brackets denote truncation at the first non-positive coefficient.
\end{proposition}
\begin{proof}
Let $d$ be the degree of any $\mathbf{m}$-free monomial in $(M_{n,\mathbf{m},k})$ and set $d' = (\sum m_i) - n - d+k$. 
First, observe from the proof of Lemma~\ref{lem:reflectionbijection} that if $d \geq d'$, or equivalently,
$d - \big((\sum m_i) - n - d\big) \geq k$,
then all $\mathbf{m}$-admissible lattice paths corresponding to $\mathbf{m}$-free monomials of degree $d$ are critical. Hence the coefficient of $t^d$ in $\HS(R/(M_{n,\mathbf{m},k}); t)$ is zero.
    On the other hand, when $d \geq d'$, the coefficient of $t^d$ in the series $(1 - t^k)\HS(R/P_{n,\mathbf{m}}; t)$ is also zero.
This follows from the fact that the Hilbert series of $R/P_{n,\mathbf{m}}$ is symmetric, with equal coefficients for $t^d$ and $t^{(\sum m_i) - n - d}$, where $(\sum m_i) - n$ is the socle degree.

    Now, consider the case $d \leq d'$. In this case, all admissible paths ending at $(n, n - d')$ are critical. By Lemma~\ref{lem:reflectionbijection}, these paths are in bijection with the critical paths ending at $(n, n - d)$. Furthermore, by Lemma~\ref{lem:path_set_up_mixed_m} and the symmetry of the Hilbert series of $P_{n,\mathbf{m}}$, they correspond bijectively to the set of all admissible paths ending at
    $\bigl(n, n - \big((\sum m_i) - n - d'\big)\bigr) = (n, n - (d - k))$.
    Therefore, the coefficient of $t^d$ in $\HS(R/(M_{n,\mathbf{m},k}); t)$ equals the difference between the coefficients of $t^d$ and $t^{d-k}$ in the Hilbert series $\HS(R/P_{n,\mathbf{m}}; t)$.
\end{proof}

\subsection{Minimal generators}

In the remainder of this section, we study some further properties of the monomial ideals $(M_{n,\mathbf{m},k})$ and their minimal generating sets.

\begin{lemma}
\label{lem:critical_path_ideal_inclusion}
If $n_1<n_2$ and $\mathbf{m}_1\in \mathbb{Z}_{\geq 2}^{n_1}$ is a prefix of $\mathbf{m}_2\in \mathbb{Z}_{\geq 2}^{n_2}$, then $(M_{n_1,\mathbf{m},k})\subseteq (M_{n_2,\mathbf{m},k})$.
\end{lemma}
\begin{proof}
It suffices to show that prolonging an $\mathbf{m}_1$-admissible critical path by $n_2 - n_1$ upward steps yields an $\mathbf{m}_2$-admissible critical path. This follows directly from Lemma~\ref{lem:Basic_reflection_arguments}(ii).
\end{proof}

\begin{lemma}\label{lem:MinGens_no_path_vertex_below}
    Let $\mathbf{P}$ be a lattice path associated with an $\mathbf{m}$-free minimal generator of $(M_{n,\mathbf{m},k})$. Then, the following statements hold.
    \begin{enumerate}
        \item[{\rm (i)}] No vertex of $\mathbf{P}$ is below the red path. 
        \item[{\rm (ii)}] Let $AB$ be the last edge of $\mathbf{P}$ and let $B'$ be the reflection of $B$.  If the slope of $AB$ is at most zero, then the edge $AB'$ has an admissible slope.
    \end{enumerate}
\end{lemma}
\begin{proof}
(i)    Consider an $\mathbf{m}$-free monomial $s$ in $(M_{n,\mathbf{m},k})$. It corresponds to a critical path $\mathbf{P}$ with $n$ steps, which contains an edge $AB$ such that $AB'$ is admissible, where $B'$ is the reflection of $B$. Suppose there exists a vertex of $\mathbf{P}$ below the red path and let $D$ be the first such vertex. Since $\mathbf{P}$ starts above the red path, there is an edge $CD$ of $\mathbf{P}$ ending at $D$.
    
    If $C$ lies on the red path, then the initial segment $\mathbf{P}_1$ of $\mathbf{P}$ up to $C$ corresponds to an $\mathbf{m}$-free monomial in $(M_{n_1,\mathbf{m},k})$, where $n_1 < n$ is the $x$-coordinate of $C$. Note that the monomial associated with $\mathbf{P}_1$ strictly divides $s$, since the slope of the edge $CD$ is not positive. By Lemma~\ref{lem:critical_path_ideal_inclusion}, $(M_{n_1,\mathbf{m},k}) \subseteq (M_{n,\mathbf{m},k})$, so $s$ cannot be a minimal generator.

    If $C$ lies above the red path, then the edge $CD$ crosses the red path, and by Lemma~\ref{lem:Basic_reflection_arguments}, the reflected edge $CD'$ is admissible. Thus, the path $\mathbf{P}_2$ that coincides with $\mathbf{P}$ up to $C$ and ends with $CD'$ corresponds to a monomial $s' \in (M_{n_2,\mathbf{m},k})$, where $n_2 \leq n$ is the $x$-coordinate of $D$ and $D'$. Since the slope of $CD'$ is greater than that of $CD$, $s'$ strictly divides $s$. By Lemma~\ref{lem:critical_path_ideal_inclusion}, $s$ cannot be minimal.

(ii)  Consider the last edge $AB$ of the path $\mathbf{P}$ associated with a minimal generator $s$ of $(M_{n,\mathbf{m},k})$, and suppose $AB$ does not go up. If $AB'$ were not admissible, then there must exist an earlier edge $EF$ in $\mathbf{P}$ with $EF'$ admissible. Let $s''$ be the monomial corresponding to the initial segment $\mathbf{P}_3$ of $\mathbf{P}$ up to $F$. Then $s'' \in (M_{n_3,\mathbf{m},k})$, where $n_3 < n$ is the $x$-coordinate of $F$. Since $AB$ has slope at most zero, $s''$ strictly divides $s$, contradicting minimality. Hence, $AB'$ must have an admissible slope.
\end{proof}

\begin{lemma}\label{lem:exponent_of_last_variable}
Let $s=x_1^{s_1}\cdots x_n^{s_n}\in (M_{n,\mathbf{m},k})$ be an $\mathbf{m}$-free monomial minimal generator with $s_n>0$ 
and write $d=\deg(s)$. Then 
\begin{equation}\label{eq:xn_degree_min_generator}
   s_n = 2d-k-\left(\sum_{i=1}^{n-1} m_i\right)+(n-1).
\end{equation}
Moreover, any $\mathbf{m}$-free monomial of degree $d$ is critical if it satisfies Equation~\eqref{eq:xn_degree_min_generator}.
\end{lemma}
\begin{proof}
Let $s \in M_{n,\mathbf{m},k}$ be a minimal generator of degree $d$ divisible by $x_n$, and let $AB$ be the last edge of its associated lattice path. By Lemma~\ref{lem:MinGens_no_path_vertex_below}(ii), the reflected edge $AB'$, with $B'$ the reflection of $B$, has admissible slope. Write $A = (n-1, n - d + \sigma)$ with $\sigma = s_n - 1$, and note that $B = (n, n - d)$ lies on or above the red line by Lemma~\ref{lem:MinGens_no_path_vertex_below}(i). Let $B' = (n, n - d - \tau)$ for some $\tau \geq 0$. The admissibility of $AB'$ implies $\sigma + \tau \leq m_n - 2$.

Since $s$ is minimal, the monomial $s/x_n$ does not lie in $M_{n,\mathbf{m},k}$. Its corresponding path ends at $\widetilde{B} = B + (0,1)$, whose reflection is  
$\widetilde{B}' = B' - (0,1)$. The edge $A\widetilde{B}'$ must be non-admissible and have slope $-(\sigma + \tau + 1)$, which implies $\sigma + \tau = m_n - 2$. Thus, the slope of $AB'$ is $-(m_n - 2)$.

To compute $s_n$, 
note from Remark~\ref{rem:Mixed_socle_and_related_degrees} that $B' = (n, n - d')$ where $d' = (\sum m_i) - n + k - d$. Since $AB'$ goes down by $m_n - 2$ steps, we have 
$A=(n-1,n-((\sum m_i)-n+k-d) + m_n-2)$. 
Hence, the vertical drop of $AB$ is
\[
n - \left(\sum m_i - n + k - d\right) + m_n - 2 - (n - d) = 2d - k - \sum_{i=1}^{n-1} m_i + (n - 1) - 1.
\]
Adding $1$ gives $s_n = 2d - k - \sum_{i=1}^{n-1} m_i + (n - 1)$, as desired.

For the final claim, consider an $\mathbf{m}$-free monomial $s$ of degree $d$ satisfying Equation \eqref{eq:xn_degree_min_generator}. The final edge $AB$ of the corresponding path then goes from $A=(n-1,(n-1)-(d-s_n))$ to $B=(n,n-d)$. Using Equation \eqref{eq:xn_degree_min_generator}, we have that $d-s_n = k + \sum_{i=1}^{n-1}m_i -(n-1)-d$. Since $B'=(n,n-d')$, the definition for $d'$ gives that the slope of $AB'$ is
\[
(n-d')-((n-1)-(d-s_n)) = 1 - d' - (k + \sum_{i=1}^{n-1}m_i -(n-1)-d) =2- m_n,
\]
which is an admissible slope. Hence $s$ corresponds to a critical path.
\end{proof}

\begin{definition}\label{def:Sets_Crit_n_m_k_j}
Let $s = x_1^{s_1} \cdots x_n^{s_n} \in R$ be a monomial, and fix an index $1 \leq j \leq n$. Define the {\em $j$-truncation} of $s$ as
$\truncs{j}= 
  x_1^{s_1} \cdots x_j^{s_j}$.
For each $1 \leq j \leq n$, define the set
\begin{equation*}\label{eq:Sets_Crit_bar_n_m_k_j}
    \overline{\mathrm{Crit}}_{n,\mathbf{m},k,j} = 
    \left\{ s \;\middle|\; s \text{ is } \mathbf{m}\text{-free},\; x_j \text{ divides } s,\;  \text{and} \ 
    s_j = 2\deg(s_{\leq j}) - k - \sum_{i=1}^{j-1} m_i + (j - 1) \right\}.
\end{equation*}
Finally, define
\begin{equation}\label{eq:Sets_Crit_n_m_k_j}
    \mathrm{Crit}_{n,\mathbf{m},k,j} = 
    \left(
        \overline{\mathrm{Crit}}_{n,\mathbf{m},k,j} 
        \setminus \left(\bigcup_{\ell=1}^{j-1} \overline{\mathrm{Crit}}_{n,\mathbf{m},k,\ell}\right)
    \right) \cap \mathbf{k}[x_1, \ldots, x_j]
\end{equation}
and
\[
\mathrm{Crit}_{n,\mathbf{m},k} = \bigcup_{j=1}^{n} \mathrm{Crit}_{n,\mathbf{m},k,j}.
\]
\end{definition}

In other words, by Lemma~\ref{lem:Basic_reflection_arguments}~(iv) and Lemma~\ref{lem:exponent_of_last_variable}, the set $\mathrm{Crit}_{n,\mathbf{m},k,j}$ consists of all monomials in $\mathbf{k}[x_1, \dots, x_j]$ that correspond to critical paths whose proper initial segments are not critical.
We are now in a position to give an alternative description of the ideal $(M_{n,\mathbf{m},k})$ as defined in Definition~\ref{def:Ideal_of_critical_paths}.

\begin{theorem}\label{thm:Gen_sys_for_critical_path_ideal}
    The ideal $(M_{n,\mathbf{m},k})$ is minimally generated by the set
    \begin{equation*}\label{eq:S_n_m_k}
        S_{n,\mathbf{m},k} =
        \{x_1^{m_1}, x_2^{m_2},\ldots,x_n^{m_n}\} \cup
        \mathrm{Crit}_{n,\mathbf{m},k}.
    \end{equation*}
where $x_j^{m_j}$ is removed if $k+\sum_{i=1}^{j-1}(m_i-1)<m_j$.
\end{theorem}

\begin{proof}
 For each $1 \leq j \leq n$, let $\mathbf{m}_j = (m_1, \ldots, m_j)$ denote the prefix of $\mathbf{m}$ of length $j$.  
By Lemma~\ref{lem:exponent_of_last_variable}, the set $\mathrm{Crit}_{n,\mathbf{m},k,j}$ consists of critical $\mathbf{m}_j$-free monomials, hence
$\mathrm{Crit}_{n,\mathbf{m},k,j} \subseteq M_{j,\mathbf{m}_j,k}$.  
Together with Lemma~\ref{lem:critical_path_ideal_inclusion}, this implies that $S_{n,\mathbf{m},k} \subseteq M_{n,\mathbf{m},k}$.  
Moreover, Lemma~\ref{lem:exponent_of_last_variable} shows that $S_{n,\mathbf{m},k}$ contains all the minimal generators of $M_{n,\mathbf{m},k}$, so we conclude that  
$(S_{n,\mathbf{m},k}) = (M_{n,\mathbf{m},k})$.

It remains to prove that $S_{n,\mathbf{m},k}$ is minimal. Fix $j \in \{1, \ldots, n\}$ and let $s \in \mathrm{Crit}_{n,\mathbf{m},k,j}$.  
Since $\max(s) = j$, by Lemma~\ref{lem:exponent_of_last_variable}, the monomial $s$ can only be divisible by a generator $t \in \mathrm{Crit}_{n,\mathbf{m},k,\ell}$ for some $\ell < j$ in order to be redundant.  
The monomial $t$ has variable with maximum index $\ell$ and satisfies the degree condition~\eqref{eq:xn_degree_min_generator} with respect to $x_\ell$. However, the definition of $s \in \mathrm{Crit}_{n,\mathbf{m},k,j}$ in \eqref{eq:Sets_Crit_n_m_k_j} excludes the existence of such a $t$.  
Finally, note that $x_j^{m_j}$ is not a minimal generator if and only if there exists some $s = x_j^{d} \in \mathrm{Crit}_{n,\mathbf{m},k,j}$ with $d < m_j$. Since $s_j = \deg(s_{\leq j}) = \deg(s) = d$, the definition of $\mathrm{Crit}_{n,\mathbf{m},k,j}$ implies that this occurs precisely when  
$d=k+\sum_{i=1}^{j-1}(m_i-1)<m_j$, 
as claimed.  
Therefore, $S_{n,\mathbf{m},k}$ is indeed the minimal generating set.
\end{proof}

\begin{example}\label{ex:Min_Gens_Lattice_Paths_(3,2,2,3)_2}
    Let $n=4$, $\mathbf{m}=(3,2,2,3)$, and $k=2$. Then we have 
    \[
    \mathrm{Crit}_{n,\mathbf{m},k}=\{ x_1^2, x_1 x_2 x_3, x_2 x_3 x_4^2, x_1 x_3 x_4^2, x_1 x_2 x_4^2 \}\,,
    \]
    and $S_{n,\mathbf{m},k}=\{ x_2^2, x_3^2, x_4^3 \}\cup \mathrm{Crit}_{n,\mathbf{m},k}$. The lattice paths of the $\mathbf{m}$-free minimal generators can be seen in Figure~\ref{fig:Lattice_paths_MinGens_(3,2,2,3)}. Note that $x_1^3$ is not a minimal generator of $(M_{n,\mathbf{m},k})$ because $k+\sum_{i=1}^0 (m_i-1)=k+0=2<3=m_1$.

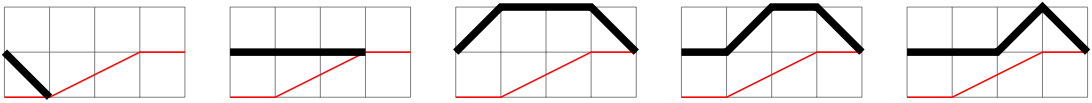
\begin{figure}[!ht]
\centering
\begin{tikzpicture}[scale=0.6]
    \begin{scope}[shift={(0,0)}]
        \draw[help lines] (0,-1) grid (4,1);
        \draw[line width=0.02cm, color=red] (0,-1)--(1,-1)--(2,-0.5)--(3,0)--(4,0);
        \draw[line width=0.1cm] (0,0)--(1,-1);
    \end{scope}

    \begin{scope}[shift={(5,0)}]
        \draw[help lines] (0,-1) grid (4,1);
        \draw[line width=0.02cm, color=red] (0,-1)--(1,-1)--(2,-0.5)--(3,0)--(4,0);
        \draw[line width=0.1cm] (0,0)--(1,0)--(2,0)--(3,0);
    \end{scope}

    \begin{scope}[shift={(10,0)}]
        \draw[help lines] (0,-1) grid (4,1);
        \draw[line width=0.02cm, color=red] (0,-1)--(1,-1)--(2,-0.5)--(3,0)--(4,0);
        \draw[line width=0.1cm] (0,0)--(1,1)--(2,1)--(3,1)--(4,0);
    \end{scope}

    \begin{scope}[shift={(15,0)}]
        \draw[help lines] (0,-1) grid (4,1);
        \draw[line width=0.02cm, color=red] (0,-1)--(1,-1)--(2,-0.5)--(3,0)--(4,0);
        \draw[line width=0.1cm] (0,0)--(1,0)--(2,1)--(3,1)--(4,0);
    \end{scope}

    \begin{scope}[shift={(20,0)}]
        \draw[help lines] (0,-1) grid (4,1);
        \draw[line width=0.02cm, color=red] (0,-1)--(1,-1)--(2,-0.5)--(3,0)--(4,0);
        \draw[line width=0.1cm] (0,0)--(1,0)--(2,0)--(3,1)--(4,0);
    \end{scope}
\end{tikzpicture}
\caption{The lattice paths of the monomials in the minimal generating set $S_{4,(3,2,2,3),2}$.}
\label{fig:Lattice_paths_MinGens_(3,2,2,3)}
\end{figure}

\end{example}

\begin{lemma}\label{lem:revlex_segment}
    Let $s \in M_{n,\mathbf{m},k}$ be a minimal generator of degree $d$. Let $t$ be any monomial of degree $d$ such that $t \succ s$ with respect to the reverse lexicographic ordering. Then $t \in (M_{n,\mathbf{m},k})$.
\end{lemma}
\begin{proof}
    Without loss of generality, we may assume that $x_n$ divides $s$. Otherwise, if the variable of largest index dividing $s$ is $x_{n_1}$ and $\mathbf{m}_1$ denotes the vector consisting of the first $n_1$ entries of $\mathbf{m}$, then $s \in M_{n_1, \mathbf{m}_1, k}$. In this case, both $s$ and all monomials $t$ of degree $d$ with $t \succ s$ lie in $\mathbf{k}[x_1,\ldots,x_{n_1}]$, reducing the problem to fewer variables.

    Hence, we may assume the last edge $AB$ in the lattice path associated with $s$ does not go up. By Lemma~\ref{lem:MinGens_no_path_vertex_below}, the endpoint $B$ lies on or above the red path. Let $B'$ be the reflection of $B$. Denote by $CD$ the last edge of the path associated with $t$.
    Since $s$ is a minimal generator, Lemma~\ref{lem:MinGens_no_path_vertex_below}(ii) guarantees that the slope of $AB'$ is admissible. If $CD = AB$, then $CD'$ is admissible, and hence $t \in (M_{n,\mathbf{m},k})$.
    Otherwise, since $t \succ s$ and $\deg(t) = \deg(s)$, the paths of $t$ and $s$ share the common endpoint $B = D = (n, n-d)$. We may write $A = (n-1, n - d + \delta)$ and $C = (n-1, n - d + \epsilon)$ for some positive integers with $\epsilon < \delta$.
    Two cases arise:

    Case 1. If $C$ lies on or below the red path, then because the path of $t$ starts above the red path, it must contain an edge $EF$ ending on or crossing the red path. For such an edge, $EF'$ (where $F'$ is the reflection of $F$) is admissible, so $t \in (M_{n,\mathbf{m},k})$.

    Case 2. If $C$ lies above the red path, then since $B'$ lies on or below the red path, the slope of $CB'$ is negative but of smaller absolute value than the negative admissible slope of $AB'$. Hence, $CB'$ is admissible, and again $t \in (M_{n,\mathbf{m},k})$.
\end{proof}

Before stating the last result of this section, we need to recall the following definition.

\begin{definition}
Let $I$ be an ideal generated by a set $V$ of $\mathbf{m}$-free monomials. Then $I$ is called \emph{strongly $\mathbf{m}$-stable} if for any $sx_j\in V$ and any index $i<j$, if $sx_i$ is $\mathbf{m}$-free, then $sx_i\in I$.
\end{definition}

Lemma~\ref{lem:revlex_segment} yields the following result, resolving a problem posed in~\cite{kling2024grobnerbasesresolutionslefschetz}.

\begin{corollary}\label{cor:m_stability}
The $\mathbf{m}$-free part of the ideal $(M_{n,\mathbf{m},k})$ is strongly $\mathbf{m}$-stable.
\end{corollary}
\begin{proof}
Let $s \in M_{n,\mathbf{m},k}$ be an $\mathbf{m}$-free minimal generator. Clearly, $s \neq 1$. Let $x_j$ be the variable of highest index dividing $s$. Then $s \in M_{j,\mathbf{m},k}$. 
Now, let $x_i$ be a variable with $i < j$ such that $s_i+1 < m_i$, and let $x_\ell$ be a variable dividing $s$ with $i < \ell$. Set $t=x_i(s/x_\ell)$.    
We need to show that $t$ has a divisor in $M_{n,\mathbf{m},k}$. It suffices to find such a divisor in $M_{j,\mathbf{m},k}$.
Since $\deg(t) = \deg(s)$ and $t \succ s$ in the reverse lexicographic ordering, the existence of this divisor follows immediately from Lemma~\ref{lem:revlex_segment}.
\end{proof}

\section{Gr\"obner bases}\label{sec:GB}

This section contains our main results. We begin by showing that the  ideal $(M_{n,\mathbf{m},k})$, generated by the monomials corresponding to critical $\mathbf{m}$-admissible lattice paths and variable powers introduced in Section~\ref{sec:InitialIdeal}, coincides with the initial ideal of $I_{n,\mathbf{m},k}$ with respect to the reverse lexicographic order induced by $x_1 \succ \cdots \succ x_n$ 
(Theorem~\ref{thm:Initial_Ideal_Equality}). We further prove that this equality holds for any term order compatible with this variable ordering (Theorem~\ref{thm:GB_determined_by_var_ranking}). In particular, $\operatorname{in}(I_{n,\mathbf{m},k}) = (M_{n,\mathbf{m},k})$. As a direct consequence, we provide a new proof that all monomial complete intersections possess the Strong Lefschetz Property over fields of characteristic zero (Corollary~\ref{cor:has_SLP}). We also describe the reduced Gr\"obner basis of $I_{n,\mathbf{m},k}$ and establish Theorem~\ref{thm:Main_Thm_Introduction} from the introduction.
In \S\ref{sub:var}, we establish an upper bound on the number of distinct reduced Gr\"obner bases of $I_{n,\mathbf{m},k}$. When the ideal $P_{n,\mathbf{m}}$ is equigenerated, we determine the exact number of such bases (Proposition~\ref{prop:Cardinality_GFan_Equigenerated}). We then study the degree sequences of the elements in these Gr\"obner bases. In particular, Proposition~\ref{prop:GB_degree_element_count_sequences_general_case} provides an efficient method for computing the degree sequences via differences of Hilbert series coefficients of $R/P_{n,\mathbf{m}}$.

\subsection{Initial ideal}
Throughout, for a monomial $p$, we write $p_i$ for the exponent of $x_i$ in $p$, so that $p = x_1^{p_1} \cdots x_n^{p_n}$.  
Our first goal is to construct, for each $s \in M_{n,\mathbf{m},k}$ that corresponds to an $\mathbf{m}$-free generator of $(M_{n,\mathbf{m},k})$, a polynomial $g_s \in I_{n,\mathbf{m},k}$ such that $\operatorname{in}(g_s) = s$.  
The key technical ingredient for this construction is provided by the following lemma.

\begin{lemma}
\label{lem:Counting_lemma}
Let $p$, $q$, and $r$ be monomials such that $pq = x_1^{\alpha_1} \cdots x_n^{\alpha_n}$ and $r|pq$.  Then
\begin{equation}\label{eq:formula}
\sum_{\substack{v| p \\ v|r}} (-1)^{\deg(v)}\left(\prod_{i=1}^n\binom{r_i}{v_i}\binom{\alpha_i-v_i}{p_i-v_i} \right) = \prod_{i=1}^n \binom{\alpha_i-r_i}{p_i}.
\end{equation}
\end{lemma}

\begin{proof}
First, consider the case $n=1$. We then have integers $p, q,$ and $r$ such that $p + q = \alpha$ and $r \leq \alpha$. The identity to prove reduces to
\begin{equation}
\label{eq:n=1_inclusion_exclusion}
\sum_{k=0}^{\infty} (-1)^k \binom{r}{k} \binom{\alpha - k}{p - k} = \binom{\alpha - r}{p}.
\end{equation}
Formally, the summation index should be restricted to $k \leq \min\{p,r\}$. However, since the terms vanish for $k > \min\{p,r\}$, extending the sum to infinity does not affect its value.

The right-hand side of \eqref{eq:n=1_inclusion_exclusion} counts the number of ways to choose $p$ elements from a set of size $\alpha$, excluding a subset of $r$ forbidden elements. The term $\binom{r}{k} \binom{\alpha - k}{p - k}$ counts the number of ways to choose exactly $k$ forbidden elements and $p-k$ from the remaining elements. The alternating sum corresponds exactly to the inclusion-exclusion principle, thereby proving the identity.

For general $n$, we apply the case $n = 1$ coordinatewise, adopting the convention that $\binom{n}{k} = 0$ whenever $k < 0$ or $k > n$.  
By interchanging the order of summation and product, the left-hand side of~\eqref{eq:formula} becomes
\[
\sum_{v} \prod_{i=1}^n (-1)^{\deg(v_i)} \binom{r_i}{v_i} \binom{\alpha_i - v_i}{p_i - v_i} 
= \prod_{i=1}^n \left( \sum_{v_i=0}^\infty (-1)^{\deg(v_i)} \binom{r_i}{v_i} \binom{\alpha_i - v_i}{p_i - v_i} \right) 
= \prod_{i=1}^n \binom{\alpha_i - r_i}{p_i},
\]
as desired.
\end{proof}

\begin{setup}\label{setup_for_main_theorem}
Let $s =x_1^{s_1}\cdots x_n^{s_n}= s' x_n^{s_n}\in R$ be a monomial satisfying the conditions
\begin{itemize}[leftmargin=2em]
    \item $s_n > 0$,
    \item $s_n = 2 \deg(s) - k - \sum_{i=1}^{n-1} m_i + (n - 1)$,
    \item $s' \mid x_1^{m_1 - 1} \cdots x_{n-1}^{m_{n-1} - 1}$.
\end{itemize}
Next, let 
$R' = \mathbf{k}[x_1, \dots, x_{n-1}, y] \big/ (x_1^{m_1}, \dots, x_{n-1}^{m_{n-1}})$
and set
$\ell = x_1 + \cdots + x_{n-1} + y$. We can then associate to $s$ two polynomials $f_s$ and $g_s$ in $R'$, defined as follows.

\begin{itemize}
\item[{\rm (i)}] \textbf{Construction of $g_s$:} Let 
\[
g_s = \sum_{s'' \mid s'} \lambda_{s''} \cdot s'' \cdot y^{\deg(s) - \deg(s'')},
\]
where for each monomial $s'' \mid s'$, we define
\[ 
\lambda_{s''} = \left( \prod_{i=1}^{n-1} \frac{s_i!}{s''_i!} \binom{m_i - s''_i - 1}{s_i - s''_i} \right) \cdot \frac{s_n!}{(\deg(s) - \deg(s''))!}.
\] 

\item[{\rm (ii)}]  \textbf{Construction of $f_s$:} Let $u$ be the unique monomial such that 
\[
u \cdot s' = x_1^{m_1 - 1} \cdots x_{n-1}^{m_{n-1} - 1}.
\]
Then define
\[
f_s = \sum_{v \mid u} \mu_v \cdot v \cdot \ell^{\deg(s) - \deg(v) - k},
\]
where for each monomial $v \mid u$, we set
\[
\mu_v = (-1)^{\deg(v)} \left( \prod_{i=1}^{n-1} \frac{s_i!}{v_i!} \binom{m_i - 1 - v_i}{u_i - v_i} \right) \cdot \frac{s_n!}{(\deg(s) - \deg(v))!}.
\]
\end{itemize}
\end{setup}

Table~\ref{table_polynomials} summarizes the key notations introduced in Setup~\ref{setup_for_main_theorem}.

\begin{table}[h]\label{table_polynomials}
\centering
\footnotesize
\setlength{\tabcolsep}{4pt}
\renewcommand{\arraystretch}{1.4}
\begin{tabular}{|p{4.54cm}|p{3cm}|p{7.1cm}|}
\hline
Polynomial definitions & Auxiliary monomials & Coefficients \\ \hline
$g_s = \sum_{s'' \mid s'} \lambda_{s''} \, s'' \, y^{\deg(s) - \deg(s'')}$ &
\parbox{3cm}{
\vspace{0.5em} $s' = \frac{s}{x_n^{s_n}}$\vspace{0.1em}} &
$\lambda_{s''} = \biggl(\prod_{i=1}^{n-1} \frac{s_i!}{s''_i!} \binom{m_i - s''_i - 1}{s_i - s''_i} \biggr) \frac{s_n!}{(\deg(s) - \deg(s''))!}$ \\ \hline

$f_s = \sum_{v \mid u} \mu_v \, v \, \ell^{\deg(s) - \deg(v) - k}$ &
\parbox{3cm}{
\vspace{0.5em}  
$u = \frac{\prod_{i=1}^{n-1} x_{i}^{m_i-1}}{s'}$
\vspace{0.5em}} &
$\mu_v = (-1)^{\deg(v)} \biggl(\prod_{i=1}^{n-1} \frac{s_i!}{v_i!} \binom{m_i - 1 - v_i}{u_i - v_i} \biggr) \frac{s_n!}{(\deg(s) - \deg(v))!}$ \\ \hline
\end{tabular}

\vspace{0.3em} 

\caption{Definitions of $f_s$ and $g_s$ 
used in Setup~\ref{setup_for_main_theorem} and Proposition~\ref{prop:Ideal_membership}.}
\label{tab:fs_gs}
\end{table}

\begin{proposition}
    \label{prop:Ideal_membership}
Let $s\in R$ be a monomial satisfying the assumptions of Setup~\ref{setup_for_main_theorem}, and $\ell = x_1 + \cdots + x_{n-1} + y$. Let $f_s, g_s \in R'$ be the associated polynomials defined therein. Then
\[
g_s = f_s \cdot \ell^k \quad \text{in } R'.
\]
\end{proposition}

\begin{proof}
For simplicity, we write $g$ and $f$ for $g_s$ and $f_s$, respectively. To prove the equality $g = f \cdot \ell^k$, we will show that both sides have the same coefficient in front of every monomial. Since both $g$ and $f \cdot \ell^k$ are homogeneous of degree $\deg(s)$, it suffices to consider monomials of that degree.

Let $\theta$ be a monomial of degree $\deg(s)$ in the variables $x_1, \dots, x_{n-1}, y$, and let $r$ be the degree of $\theta$ in $x_1, \dots, x_{n-1}$. Inserting the definition of $\lambda_{s''}$ into that of $g$, we find that the coefficient of $\theta$ in $g$ is
\begin{equation}
\label{eq: g_coeff_of_theta}
[\theta]\, g = \frac{s_n!}{(\deg(s)-r)!} \left( \prod_{i=1}^{n-1} \frac{s_i!}{\theta_i!} \binom{m_i - 1 - \theta_i}{s_i - \theta_i} \right).
\end{equation}
We now compute the coefficient of $\theta$ in $f \cdot \ell^k$ and show that it matches the coefficient of $\theta$ in $g$.
Write $\ell = \overline{\ell} + y$, where $ \overline{\ell} =x_1+\cdots+x_{n-1}$. Then
\begin{align*}
f \cdot \ell^k &= \sum_{v \mid u} \mu_v \cdot v \cdot (\overline{\ell} + y)^{\deg(s) - \deg(v)} \\
&= \sum_{v \mid u} \mu_v \cdot v \cdot \sum_{i=0}^{\deg(s)-\deg(v)} \binom{\deg(s) - \deg(v)}{i} \cdot \overline{\ell}^i \cdot y^{\deg(s) - \deg(v) - i}.
\end{align*}
Since $\theta$ has degree $\deg(s)-r$ in $y$, we extract the coefficient of $y^{\deg(s) - r}$ from $f \cdot \ell^k$. Let $\mathbf{x}^{\mathbf{m}-1} = x_1^{m_1 - 1} \cdots x_{n-1}^{m_{n-1} - 1}$. Using the multinomial theorem and the fact that we are working in the quotient ring $\mathbf{k}[x_1, \dots, x_{n-1}, y]/(x_1^{m_1}, \dots, x_{n-1}^{m_{n-1}})$, we get
\begin{align*}
[y^{\deg(s)-r}] (f \cdot \ell^k) &= \sum_{v \mid u} \mu_v \cdot v \cdot \binom{\deg(s) - \deg(v)}{r - \deg(v)} \cdot \overline{\ell}^{r - \deg(v)} \\
&= \sum_{v \mid u} \mu_v \cdot v \cdot \binom{\deg(s) - \deg(v)}{r - \deg(v)} \sum_{\substack{w \mid \mathbf{x}^{\mathbf{m}-1} \\ \deg(w) = r - \deg(v)}} \frac{(r - \deg(v))!}{w_1! \cdots w_{n-1}!} \cdot w \\
&= \frac{s_1! \cdots s_n!}{(\deg(s) - r)!} \sum_{v \mid u} \sum_{\substack{w \\ vw \mid \mathbf{x}^{\mathbf{m}-1} \\ \deg(vw) = r}} (-1)^{\deg(v)} \left( \prod_{i=1}^{n-1} \frac{1}{v_i! w_i!} \binom{m_i - 1 - v_i}{u_i - v_i} \right) \cdot vw,
\end{align*}
where in the last equality we used the definition of $\mu_v$.
Therefore, the coefficient of $\theta$ in $f \cdot \ell^k$ is
\begin{align*}
[\theta] (f \cdot \ell^k) &= \frac{s_1! \cdots s_n!}{\theta_1! \cdots \theta_{n-1}! (\deg(s) - r)!} \sum_{\substack{v \mid u \\ v \mid \theta}} (-1)^{\deg(v)} \left( \prod_{i=1}^{n-1} \frac{\theta_i!}{v_i!(\theta_i - v_i)!} \binom{m_i - 1 - v_i}{u_i - v_i} \right) \\
&= \frac{s_1! \cdots s_n!}{\theta_1! \cdots \theta_{n-1}! (\deg(s) - r)!} \sum_{\substack{v \mid u \\ v \mid \theta}} (-1)^{\deg(v)} \left( \prod_{i=1}^{n-1} \binom{\theta_i}{v_i} \binom{m_i - 1 - v_i}{u_i - v_i} \right).
\end{align*}
Now, applying Lemma~\ref{lem:Counting_lemma} with $r_i = \theta_i$, $\alpha_i = m_i - 1$, and $p_i = u_i$, we obtain
\[
[\theta] (f \cdot \ell^k) = \frac{s_1! \cdots s_n!}{\theta_1! \cdots \theta_{n-1}! (\deg(s) - r)!} \prod_{i=1}^{n-1} \binom{m_i - 1 - \theta_i}{u_i}.
\]
Finally, since $\binom{m_i - 1 - \theta_i}{u_i} = \binom{m_i - 1 - \theta_i}{s_i - \theta_i}$, this equals $[\theta]\, g$ from~\eqref{eq: g_coeff_of_theta}, completing the proof.
\end{proof}

\begin{remark}\label{rem:incarnations_of_g_s}
    Consider the polynomials $f_s$, $g_s$, and $\ell=x_1+\cdots +x_{n-1}+y$ in the ring $R'=\mathbf{k}[x_1, \dots, x_{n-1}, y] \big/ (x_1^{m_1}, \dots, x_{n-1}^{m_{n-1}})$ as defined in Setup~\ref{setup_for_main_theorem}. For any integer $\overline{n}\geq n$, consider the injective ring homomorphism
    \begin{equation}\label{eq:injective_ring_homomorphism}
        \phi_{\overline{n}}:\mathbf{k}[x_1,\ldots,x_{n-1},y]\to\mathbf{k}[x_1,\ldots,x_{n-1},x_n,\ldots,x_{\overline{n}}]
    \end{equation}
    defined by $x_i\mapsto x_i$ for $1\leq i<n$ and $y\mapsto x_n+\cdots +x_{\overline{n}}$. Then, since $g_s=f_s\cdot \ell^{k}$ in $R'$ and $\phi_{\overline{n}}(\ell)=x_1+\cdots + x_{\overline{n}}$, it follows that $\phi_{\overline{n}}(g_s)\in I_{\overline{n},\overline{\mathbf{m}},k}$ for any $\overline{\mathbf{m}}$ obtained by extending $(m_1,\ldots,m_{n-1})$ with positive integers $m_{n}, \ldots, m_{\overline{n}}$. For simplicity, we also denote $\phi_{\overline{n}}(g_s)$ by $g_s$.
\end{remark}

We now have all the necessary ingredients to prove our first main result, that the initial ideal of $I_{n,\mathbf{m},k}$ coincides with the ideal $(M_{n,\mathbf{m},k})$ defined via critical $\mathbf{m}$-admissible lattice paths.

\begin{theorem}\label{thm:Initial_Ideal_Equality}
We have $\ini(I_{n,\mathbf{m},k})=(M_{n,\mathbf{m},k})$ with respect to the reverse lexicographic ordering.
\end{theorem}

\begin{proof}
We begin by observing that the Hilbert series of $R/I_{n,\mathbf{m},k}$ and $R/\ini(I_{n,\mathbf{m},k})$ coincide. The coefficientwise smallest Hilbert series that $R/I_{n,\mathbf{m},k}$ can attain is 
    $\big[(1 - t^k)\HS(R/P_{n,\mathbf{m}}; t)\big]$,
where the square brackets denote truncation at the first non-positive coefficient.
From Proposition~\ref{prop:Counting_Critical_Paths}, we know that this series coincides with $\HS(R/(M_{n,\mathbf{m},k});t)$. 

Now, let $s \in \mathrm{Crit}_{n,\mathbf{m},k}$ be any $\mathbf{m}$-free generator. Every such $s$ satisfies the degree condition of Proposition~\ref{prop:Ideal_membership} with respect to the variable of maximal index dividing it, and thus there is an ideal element $g_s \in I_{n,\mathbf{m},k}$ as in Remark~\ref{rem:incarnations_of_g_s} associated with it. But $\ini(g_s)=s$ by the definition of the degree reverse lexicographic ordering. Thus, $(M_{n,\mathbf{m},k})\subseteq \ini(I_{n,\mathbf{m},k})$, and $\HS(R/(\ini(I_{n,\mathbf{m},k})));t)$ is coefficientwise bounded above by the series $\HS(R/(M_{n,\mathbf{m},k});t)=[(1-t^k)\HS(R/P_{n,\mathbf{m}};t)]$. The minimality property of this series now implies that 
$\HS(R/\ini(I_{n,\mathbf{m},k}); t)$ equals 
$\big[(1 - t^k) \HS(R/P_{n,\mathbf{m}}; t)\big]$, 
which in turn establishes the claim 
$\ini(I_{n,\mathbf{m},k}) = (M_{n,\mathbf{m},k})$.
\end{proof}

With this, we can now give a new proof that all monomial complete intersection have the SLP over fields of characteristic zero.

\begin{corollary}\label{cor:has_SLP}
Let $\mathbf{k}$ be a field of characteristic zero. Then the monomial complete intersection $\mathbf{k}[x_1,\dots, x_n]/(x_1^{m_1},\dots, x_n^{m_n})=R/P_{n,\mathbf{m}}$ has the strong Lefschetz property.
\end{corollary}

\begin{proof}
Recall that $A=R/P_{n,\mathbf{m}}$ has the SLP if the multiplication maps
\[
\mu_{\ell,d,k} \colon A_d \to A_{d+e}, \quad f \mapsto \ell^k \cdot f
\]
have full rank for all $d,k\geq 0$ and $\ell=x_1+\cdots + x_n$. This is equivalent to $A/(\ell^k)$ having Hilbert series $\big[(1 - t^k) \, \HS(A; t)\big]$ for all $k \geq 0$, where the brackets denote truncation at the first non-positive term.
 Using Proposition \ref{prop:Counting_Critical_Paths} and Theorem \ref{thm:Initial_Ideal_Equality}, we have
\begin{align*}
\HS(A/(\ell^k);t) &= \HS(R/I_{n,\mathbf{m},k};t)
= \HS(R/(\ini(I_{n,\mathbf{m},k}));t)\\ 
&= \HS(R/(M_{n,\mathbf{m},k});t)
= \big[(1 - t^k) \, \HS(R/P_{n,\mathbf{m}}; t)\big],
\end{align*}
showing that $A$ has the SLP.
\end{proof}

Having established that $R/P_{n,\mathbf{m}}$ satisfies the SLP, a natural question is whether the algebras $R/I_{n,\mathbf{m},k}$ or $R/\operatorname{in}(I_{n,\mathbf{m},k})$ also have the SLP. However, this is not true in general. Indeed, fix integers $n\geq 4$ and $m=k\geq 2$. After a generic change of variables, we have an isomorphism
$R/I_{n,\mathbf{m},m} \cong R/(\ell_1^m, \dots, \ell_{n+1}^m)$ for some generic linear forms $\ell_i$. It is known that these algebras fail the WLP for all $n \geq 4$ except for when $(n,m)\in \{(4,2),(5,2),(5,3),(7,2)\}$, in which case it has the WLP \cite[Theorem 1.1]{WLP_n+1_forms}. Hence, the same holds for $R/I_{n,\mathbf{m},m}$. Using a result of Conca \cite[Theorem 1.1]{Conca}, which states that if $R/\operatorname{in}(I_{n,\mathbf{m},m})$ has the WLP, then so does $R/I_{n,\mathbf{m},m}$, we deduce that $R/\operatorname{in}(I_{n,\mathbf{m},m})$ also fails the WLP for the same parameters where $R/I_{n,\mathbf{m},m}$ fails.
Finally, a direct computation in Macaulay2 confirms that $R/\operatorname{in}(I_{n,\mathbf{m},m})$ has the WLP for $(n,m) \in \{(4,2), (5,2), (5,3), (7,2)\}$.

Our next goal is to determine the reduced Gr\"obner basis of $I_{n,\mathbf{m},k}$. To achieve this, we analyze the non-leading terms of the polynomials $g_s$. Here, $I_{\overline{n},\overline{\mathbf{m}},k}$ denotes the ideal associated with any pair $(\overline{n},\overline{\mathbf{m}})$ where $\overline{n} \geq n$ and $\overline{\mathbf{m}}$ is obtained from $\mathbf{m}$ by appending positive integers $m_{n+1}, \ldots, m_{\overline{n}}$.

\begin{lemma}\label{lem:Tail_terms_closed_form_GB_polys}
    Let $g_s$ be as defined in Setup~\ref{setup_for_main_theorem}(i). Then $\mathrm{in}(g_s) = s$, and all non-leading terms of $g_s$ lie outside $\mathrm{in}(I_{\overline{n},\overline{\mathbf{m}},k})$.
\end{lemma}
\begin{proof}
    From the definition of the reverse lexicographic ordering, it is clear that $\mathrm{in}(g_s) = s$.

Next, we analyze the non-leading terms of $g_s$. For the sake of contradiction, suppose that one of them, say $t$, lies in the initial ideal. Consider the monomial $t_{\leq n}$, obtained by setting the variables $x_{n+1}, \ldots, x_{\overline{n}}$ equal to $1$ in $t$. By construction, $t_{\leq n}$ is a strict divisor of $s$. Since $s$ is a minimal generator of $M_{n,\mathbf{m},k}$, it follows that $t_{\leq n}$ is not in the ideal $\mathrm{in}(I_{n,\mathbf{m},k})$.

Since $t$ is in the initial ideal, this implies that we can initiate a reflection of the lattice path corresponding to $t$ at some edge $AB$ situated between the vertical lines $x = n + \delta$ and $x = n + \delta + 1$ for some $\delta \geq 0$, where $AB$ is the first such edge. Recall that this reflection yields an edge $AB'$ with admissible slope, where $B'$ is the reflection of $B$. Moreover, in this case we have $t_{n+\delta+1}>0$.

The $y$-coordinate of the point $B$ is given by $n + \delta + 1 - \deg(t_{\leq n+\delta+1})$. This quantity represents the number of vertical steps taken by the lattice path before reaching the vertical line $x = n+\delta+1$. 
Since $\deg(t_{\leq n+\delta+1}) \leq d=\deg(s)$, we obtain that the $y$-coordinate of $B$ is at least $n + \delta + 1 - d$. The point $A$ lies to the left of $B$, and by the definition of admissible slopes, the $y$-coordinate of $A$ is not lower than one less that of $B$. 
Moreover, the $y$-coordinate of the point on the red path corresponding to $B$ is given by
$\frac{3(n+\delta+1) - \left(\sum_{i=1}^{n+\delta+1} m_i\right) - k}{2}$
by Remark~\ref{rem:Mixed_socle_and_related_degrees}. Hence, the vertical distance between $B$ and its reflection $B'$ is at least $\epsilon$ where
\begin{eqnarray*}
\epsilon &=& 2 \cdot \left( \frac{2n + 2\delta + 2 - 2d - 3(n + \delta + 1) + \sum_{i=1}^{n+\delta+1} m_i + k}{2} \right) \\
         &=& -\delta + \left( -2d + k + \sum_{i=1}^{n-1} m_i - (n - 1) \right) + m_n - 2 + \sum_{i=n+1}^{n+\delta+1} m_i \\
         &\geq& -\delta + \sum_{i=n+1}^{n+\delta+1} m_i = -\delta + \sum_{i=n+1}^{n+\delta} m_i + m_{n+\delta+1},
\end{eqnarray*}
where the last inequality holds because
$-2d + k + \sum_{i=1}^{n-1} m_i - (n - 1) + m_n - 2 \geq 0$,
which follows from the expression for $s_n$ given in~\eqref{eq:xn_degree_min_generator}, combined with the condition $s_n \leq m_n - 2$.

Since each $m_i > 0$, we have $\epsilon \geq m_{n+\delta+1}$. By construction, $\epsilon - 1$ provides a lower bound for the absolute value of the negative slope of the segment $AB'$. Consequently, the slope is at most $1-\epsilon \leq 1 - m_{n+\delta+1}$, which is not admissible, leading to the desired contradiction. Therefore, we conclude that
$t \notin \mathrm{in}(I_{\overline{n}, \overline{\mathbf{m}}, k})$,
as claimed.
\end{proof}

\subsection{Reduced Gr\"{o}bner basis}
We now describe the reduced Gr\"obner basis of $I_{n,\mathbf{m},k}$ by restating and proving the main part of Theorem~\ref{thm:Main_Thm_Introduction} from the introduction.

\begin{theorem}\label{thm:Reduced_GB}
The reduced Gr\"{o}bner basis of $I_{n,\mathbf{m},k}$ with respect to the reverse lexicographic order with $x_1 \succ \cdots \succ x_n$ is given by
\begin{equation}\label{eq:Reduced_GB}
    G_{n,\mathbf{m},k} = \{ x_1^{m_1}, \ldots, x_n^{m_n} \} \cup \bigcup_{j=1}^n \{ g_s \mid s \in \mathrm{Crit}_{n,\mathbf{m},k,j} \}\,,
\end{equation}
where $x_j^{m_j}$ is removed if $k+\sum_{i=1}^{j-1}(m_i-1)<m_j$. Here, for any monomial $s = s' x_j^{s_j}$ in $\mathrm{Crit}_{n,\mathbf{m},k,j}$, the corresponding polynomial $g_s$ is 
\begin{equation}\label{eq:more_explicit_gb_polynomial_g_s}
    g_s = \sum_{s'' \mid s'} \lambda_{s''} \, s'' (x_j + \cdots + x_n)^{\deg(s) - \deg(s'')},
\end{equation}
where
\[
 \lambda_{s''} = \left( \prod_{i=1}^{j-1} \frac{s_i!}{s''_i!} \binom{m_i - s''_i - 1}{s_i - s''_i} \right) \cdot \frac{s_j!}{(\deg(s) - \deg(s''))!}
\]
and where we reduce the monomials in the expansion of $(x_j + \cdots + x_n)^{\deg(s) - \deg(s'')}$ modulo $(x_j^{m_j},\dots,x_n^{m_n})$ if necessary. 
\end{theorem}

\begin{proof}
By Theorem~\ref{thm:Initial_Ideal_Equality}, we have $\ini(I_{n,\mathbf{m},k}) = (M_{n,\mathbf{m},k})$. Moreover, Theorem~\ref{thm:Gen_sys_for_critical_path_ideal} shows that $(M_{n,\mathbf{m},k})$ is minimally generated by the selected pure powers together with the monomials $s$ satisfying $s \in \mathrm{Crit}_{n,\mathbf{m},k,j}$ for some $j$. 
For each $s \in \mathrm{Crit}_{n,\mathbf{m},k,j}$, the corresponding polynomial $g_s$ is monic with leading monomial $s$, and by Lemma~\ref{lem:Tail_terms_closed_form_GB_polys}, all its tail terms lie outside the initial ideal $\ini(I_{n,\mathbf{m},k})$. 
This verifies that the collection $\{g_s \mid s \in \mathrm{Crit}_{n,\mathbf{m},k,j},\ j=1,\dots,n\}$, together with the selected pure powers, forms the reduced Gr\"obner basis of $I_{n,\mathbf{m},k}$.
\end{proof}

\subsection{Variations of Gr\"obner bases}\label{sub:var}
Here, we present an alternative description of the polynomials in the reduced Gr\"obner basis $G_{n,\mathbf{m},k}$ defined in \eqref{eq:Reduced_GB}. We show that this basis depends solely on the underlying variable order, rather than on the particular monomial ordering chosen. Moreover, we explore bounds on the number of distinct reduced Gr\"obner bases of the ideal~$I_{n,\mathbf{m},k}$.

\begin{proposition}\label{prop:GBPolys_Support_and_coeffs}
Let $s$ be an $\mathbf{m}$-free monomial in $M_{n,\mathbf{m},k}$ and assume $s_n\neq 0$.  
Then there exists a \emph{unique} polynomial $g_s \in G_{\overline{n}, \overline{\mathbf{m}}, k}$ with leading term $s$, given by
\begin{equation}\label{eq:GBPoly_alternative_form}
		g_s=s+\sum_{t\in \mathrm{Tail}(s)} \lambda_t\cdot  t\, ,
	\end{equation}
where the tail terms and coefficients are defined by
\[
\begin{aligned}
    \mathrm{Tail}(s) = \big\{\, t \in \mathrm{Mon}(R) : \;&
    t \text{ is } \overline{\mathbf{m}}\text{-free}, \quad \deg(t) = \deg(s), \quad s \succ t, \\
    & t \notin \mathrm{in}(I_{\overline{n}, \overline{\mathbf{m}}, k}) \text{ and }
    t_i \leq s_i \quad  \forall i < n 
     \big\},
\end{aligned}
\]
and for each $t \in \mathrm{Tail}(s)$,
\[
\lambda_t =  
\left(\prod_{i=1}^{n-1} \binom{m_i - t_i - 1}{s_i - t_i}\right) \cdot \frac{s_1! \cdots s_n!}{t_1! \cdots t_n!}.
\]
\end{proposition}

\begin{proof}
Let $g_s$ denote the Gr\"obner basis polynomial defined in Equation~\eqref{eq:more_explicit_gb_polynomial_g_s}, and let $\overline{g}_s$ be the polynomial given by Equation~\eqref{eq:GBPoly_alternative_form}. We need to show that $g_s = \overline{g}_s$.

We begin by showing that $g_s$ and $\overline{g}_s$ have the same support. Take any monomial $t$ in the support of $g_s$. By construction, $t$ is $\overline{\mathbf{m}}$-free and has the same degree as $s$. Moreover, Lemma~\ref{lem:Tail_terms_closed_form_GB_polys} guarantees that $t \prec s$ and that $t \notin \mathrm{in}(I_{\overline{n}, \overline{\mathbf{m}}, k})$. 
Finally, if $x_i$ divides $t$ for some $i < \max(s) = n$, then writing $t = s'' q$, where $q$ involves only variables with indices at least $n$ and $s''$ divides $s$, implies $t_i = s''_i \leq s_i$. Therefore, $t$ lies in the support of $\overline{g}_s$.

Conversely, let $t$ be a monomial in the support of $\overline{g}_s$. Write $t = p q$ with 
$p \in \mathbf{k}[x_1, \ldots, x_{n-1}]$ and $q \in \mathbf{k}[x_n, \ldots, x_{\overline{n}}]$. 
The defining condition on $t$ implies that $p$ divides 
$s' = s / x_n^{s_n}$. Hence, we can write $p = s''$ for some monomial $s''$ dividing $s'$. 
Since $t$ and $s$ have the same total degree, it follows that $t = p q$ appears as a monomial in the expansion of $g_s$ corresponding to the choice $p = s''$. Therefore, the supports of $g_s$ and $\overline{g}_s$ coincide.

\medskip

Next, consider a monomial $t = p q$ as above. In $\overline{g}_s$, its coefficient is
\[
\left(\prod_{i=1}^{n-1} \binom{m_i - t_i - 1}{s_i - t_i} \right) \frac{s_1! \cdots s_{n}!}{t_1! \cdots t_{\overline{n}}!}
= \left(\prod_{i=1}^{n-1} \binom{m_i - p_i - 1}{s_i - p_i} \right) \frac{s_1! \cdots s_{n}!}{p_1! \cdots p_{n-1}! \, q_n! \cdots q_{\overline{n}}!}.
\]
In $g_s$, the coefficient of $t$ factors as $\mu_1 \mu_2$ where $\mu_1 = \lambda_{s''}$ and $s'' = p$. Since $\deg(s) - \deg(p)=\deg(q)$, we thus get
\[
\mu_1 = \left(\prod_{i=1}^{n-1} \binom{m_i - p_i - 1}{s_i - p_i} \right) \frac{s_1! \cdots s_{n}!}{p_1! \cdots p_{n-1}! \, \deg(q)!}.
\]
The factor $\mu_2$ is the coefficient of $q$ in the expansion of $(x_n + \cdots + x_{\overline{n}})^{\deg(q)}$, which by the multinomial theorem equals
\[
\mu_2 = \binom{\deg(q)}{q_n, \ldots, q_{\overline{n}}} = \frac{\deg(q)!}{q_n! \cdots q_{\overline{n}}!}.
\]
Thus, the coefficient of $t$ in $g_s$ is the product
\[
\mu_1 \mu_2 = \left(\prod_{i=1}^{n-1} \binom{m_i - p_i - 1}{s_i - p_i} \right) \frac{s_1! \cdots s_{n}!}{p_1! \cdots p_{n-1}! \, \deg(q)!} \cdot \frac{\deg(q)!}{q_n! \cdots q_{\overline{n}}!},
\]
which exactly matches the coefficient of $t$ in $\overline{g}_s$, as desired.
\end{proof}

\begin{remark}\label{rem:prime_factors_gb_coeffs}
The coefficients of the Gr\"obner basis polynomials $g_s \in G_{n,\mathbf{m},k}$ are rational numbers whose reduced numerators and denominators have prime factors only among those less than or equal to $\max\{m_1, \ldots, m_n\}$. In particular, this result resolves \cite[Conjecture 5.1]{kling2024grobnerbasesresolutionslefschetz}.
\end{remark}

In Theorem~\ref{thm:Reduced_GB}, we explicitly described the reduced Gr\"obner basis of $I_{n,\mathbf{m},k}$ with respect to the reverse lexicographic order. The following theorem shows that this same set forms the reduced Gr\"obner basis of $I_{n,\mathbf{m},k}$ for any term order on $R$ that respects the variable ordering $x_1 \succ \cdots \succ x_n$, thus providing the last part of the proof of Theorem \ref{thm:Main_Thm_Introduction}.

\begin{theorem}\label{thm:GB_determined_by_var_ranking}
   Let $\prec$ be any term order on $R$ satisfying $x_1 \succ \cdots \succ x_n$. Then the reduced Gr\"obner basis of $I_{n,\mathbf{m},k}$ with respect to $\prec$ coincides exactly with $G_{n,\mathbf{m},k}$ defined in \eqref{eq:Reduced_GB}. In particular, for any such term order, the initial ideal $\ini(I_{n,\mathrm{m},k})$ is minimally generated by
\begin{equation*}
        \{x_1^{m_1}, x_2^{m_2},\ldots,x_n^{m_n}\} \cup
        \mathrm{Crit}_{n,\mathbf{m},k}.
    \end{equation*}
where $x_j^{m_j}$ is removed if $k+\sum_{i=1}^{j-1}(m_i-1)<m_j$. 
\end{theorem}
\begin{proof}
Let $g_s \in G_{n,\mathbf{m},k}$ be a polynomial in the reduced Gr\"obner basis of $I_{n,\mathbf{m},k}$ with respect to the degree reverse lexicographic order. Its leading monomial is $s$. By Proposition~\ref{prop:GBPolys_Support_and_coeffs}, any non-leading monomial $t$ in the support of $g_s$ satisfies that the Laurent monomial $t/s$, in its reduced form, can be written as $a/b$, where only variables with indices at least $\max(s)$ appear in the numerator $a$, and only variables with indices less than $\max(s)$ appear in the denominator $b$.
Since the term ordering $\prec$ respects the variable ordering $x_1 \succ \cdots \succ x_n$, we have $a \prec b$. Writing $c = \gcd(s,t)$, it follows that 
$t = (t/a) a = c a \prec c b = (s/b) b = s$. Hence, $t \prec s$, so $s$ remains the leading monomial of $g_s$ with respect to $\prec$.

Consequently, the initial ideal of $I_{n,\mathbf{m},k}$ with respect to the degree reverse lexicographic order is contained in the initial ideal with respect to $\prec$. Since their Hilbert functions coincide, these initial ideals must be equal. Applying Lemma~\ref{lem:Tail_terms_closed_form_GB_polys}, we conclude that $G_{n,\mathbf{m},k}$ is also the reduced Gr\"obner basis of $I_{n,\mathbf{m},k}$ with respect to $\prec$.
\end{proof}

\begin{example}\label{ex:Reduced_GB_(3,2,2,3)_2}
    Let $n=4$, $\mathbf{m}=(3,2,2,3)$, and $k=2$, as in Example~\ref{ex:Min_Gens_Lattice_Paths_(3,2,2,3)_2}. With respect to any monomial ordering $\prec$ with $x_1\succ x_2\succ x_3\succ x_4$, we have the reduced Gr\"{o}bner basis
    \begin{align*}
        G_{n,\mathbf{m},k}
        =
        \big\{
        & x_3^{2},\; x_2^{2},\; x_4^{3},\; x_1^{2}+2 x_1 x_2+2 x_1 x_3+2 x_2 x_3+2 x_1 x_4+2 x_2 x_4+2 x_3 x_4+x_4^{2},\\
        & x_1 x_2 x_3 + x_1 x_2 x_4 + x_1 x_3 x_4 + 2 x_2 x_3 x_4 + \frac{1}{2}x_1 x_4^{2} +  x_2 x_4^{2} + x_3 x_4^{2},\\
        & x_2 x_3 x_4^{2},\; x_1 x_3 x_4^{2},\; x_1 x_2 x_4^{2}
        \big\}.
    \end{align*}
\end{example}

\begin{remark}\label{rem:Almost_CI_is_CI}
Note that the ideals of the form $I_{n,\mathbf{m},k}$ for which we can determine their Gr\"obner basis not only contains almost complete intersections, but also any complete intersection of the form
\[
\overline{I}_{n,\mathbf{m},k} = (x_1^{m_1},\dots, x_{n-1}^{m_{n-1}}, (x_1+\cdots + x_n)^k).
\]
Indeed, the socle degree of $R/\overline{I}_{n,\mathbf{m},k}$ is $\sum_{i=1}^{n-1}(m_i-1) + (k-1)$, so if $m_n$ is larger than that, then $x_n^{m_n}\in \overline{I}_{n,\mathbf{m},k}$ and  $I_{n,\mathbf{m},k}=\overline{I}_{n,\mathbf{m},k}$.
\end{remark}

\begin{corollary}\label{cor:At_most_n_factorial_GBs}
    The ideal $I_{n,\mathbf{m},k}$ has at most $n!$ distinct reduced Gr\"{o}bner bases.
\end{corollary}
\begin{proof}
Every term ordering induces a total ordering on the variables $x_1, \ldots, x_n$. Since the reduced Gr\"obner basis $G_{n,\mathbf{m},k}$ depends only on the variable ordering (by Theorem~\ref{thm:GB_determined_by_var_ranking}), the number of distinct reduced Gr\"obner bases is bounded above by $n!$.
\end{proof}

In the special case where the ideal $P_{n,\mathbf{m}}$ is equigenerated in degree $m \geq 3$, that is, when $\mathbf{m} = (m, \ldots, m)$, we can precisely determine the number of distinct reduced Gr\"obner bases of $I_{n,\mathbf{m},k}$. For the case $m=2$, this number is given by the multinomial coefficient $\binom{n}{k, 2, \ldots, 2, 1}$, where the final entry $1$ appears if and only if $n - k \equiv 1 \pmod{2}$; see \cite[Proposition 2.29]{kling2024grobnerbasesresolutionslefschetz}.

\begin{proposition}\label{prop:Cardinality_GFan_Equigenerated}
Suppose $\mathbf{m} = (m, \ldots, m)$ with $m \geq 3$. Then the ideal $I_{m,\mathbf{m},k}$ has exactly
$$(n!)/(\min\{\lfloor k/(m-1) \rfloor,n\}!)$$ 
distinct reduced Gr\"{o}bner bases.
\end{proposition}

\begin{proof}
Set $\mu = \lfloor k/(m-1) \rfloor$. We first show that $I_{n,\mathbf{m},k}$ has at most $(n!)/(\min\{\mu,n\}!)$ distinct reduced Gr\"{o}bner bases. Recall that the socle degree of the algebra $R/P_{j,\mathbf{m}_j}$ is $ D_j = j \cdot (m-1)$.

Consider first the case $\mu \geq n$, i.e., $\min\{\mu,n\} = n$. Then $k \geq n(m-1) = D_n$, so either $I_{n,\mathbf{m},k} = P_{n,\mathbf{m}}$ or the $\mathbf{m}$-free part of $I_{n,\mathbf{m},k}$ is generated in the socle degree of the Gorenstein algebra $A_{n,\mathbf{m}}$. In either situation, it is clear that there is only one reduced Gr\"obner basis.

Now consider the case $\mu < n$, i.e., $\min\{\mu,n\} = \mu$. Since $k \geq \mu (m-1) = D_\mu$, the ideal $\ini(I_{n,\mathbf{m},k}) \cap \mathbf{k}[x_1,\ldots,x_\mu]$ contains at most one $\mathbf{m}$-free monomial, $s_{\mu}=x_1^{m-1}\cdots x_\mu^{m-1}$. In turn, $\overline{\mathrm{Crit}}_{m,\mathbf{m},k,j}$ contains no monomials of degree lower than $\deg(s_\mu)$, for $1 \leq j \leq \mu$. This implies, for any monomial ordering $\prec$ that ranks the variables $x_1,\ldots,x_\mu$ higher than the remaining variables, that $\ini(I_{n,\mathbf{m},k})$ only depends on the ranking of $x_{\mu + 1},\ldots,x_n$. As a consequence, the reduced Gr\"{o}bner basis is invariant under permutation of the first $\mu$ variables, and the number of distinct reduced Gr\"{o}bner bases is thus at most $(n!)/(\mu!)$.

Now we show that $I_{n,\mathbf{m},k}$ has at least $(n!)/(\min\{\mu,n\}!)$ distinct reduced Gr\"{o}bner bases. The case $\min\{\mu,n\}=n$ is easy, as there always exists at least one reduced Gr\"{o}bner basis. In the case $\min\{\mu,n\}=\mu$, recall from the previous paragraph that $\overline{\mathrm{Crit}}_{m,\mathbf{m},k,j}$ contains no monomials of degree lower than $\deg(s_\mu)$, for $1 \leq j \leq \mu$. Now consider two monomial orderings $\prec_a$ and $\prec_b$ that rank the variables $x_1,\ldots,x_\mu$ higher than the rest, and such that their induced variable rankings differ but coincide up to a certain variable $x_\ell$, modulo a permutation of the first $\mu$ variables. Without loss of generality, the rankings are $x_1 \succ_a \cdots \succ_a x_\mu \succ_a x_{\mu +1} \succ_a \cdots \succ_a x_\ell \succ_a x_a \succ_a \cdots$ and $x_1 \succ_b \cdots \succ_b x_\mu \succ_b x_{\mu +1} \succ_b \cdots \succ_b x_\ell \succ_b x_b \succ_b \cdots$, with $x_a \neq x_b$. Since $k < (\mu+1)\cdot(m-1)$, for both monomial orderings there must be an $\mathbf{m}$-free minimal generator of $\ini(I_{n,\mathbf{m},k})$ that involves only the first $\mu+1$ variables. We will use this fact to show that $\ini_{\prec_a}(I_{n,\mathbf{m},k})$ contains a minimal generator with lowest-ranked variable $x_a$, and similarly for $\ini_{\prec_b}(I_{n,\mathbf{m},k})$ with $x_b$. An inspection of the involved lattice paths (with respect to enumerations adapted to the individual variable rankings) then shows that $\ini_{\prec_a}(I_{n,\mathbf{m},k}) \neq \ini_{\prec_b}(I_{n,\mathbf{m},k})$.

We remark that in the special case $k=1$, we have $\mu=0$. In this case, if $\prec_a$ and $\prec_b$ differ already in their highest, second highest, or third highest ranked variable, then $\ini_{\prec_a}(I_{n,\mathbf{m},k}) \neq \ini_{\prec_b}(I_{n,\mathbf{m},k})$ by the following argument. If they differ already at the highest ranked variable, then the only degree one generator in a reduced Gr\"{o}bner basis must have the highest ranked variable as leading monomial, exhibiting their difference. If they differ at the second or third highest ranked variable, then an application of Lemma~\ref{lem:exponent_of_last_variable} gives that the monomial $x_c^{m-1}x_d$, where $x_c$ and $x_d$ are the second and third highest ranked variables respectively, is a minimal generator of the initial ideal but $x_cx_d^{m-1}$ is not, proving their difference.

Thus it suffices to show that, for a given variable enumeration, the existence of a minimal generator $s$ of the initial ideal with $\max(s)=\nu$ implies the existence of a minimal generator $\overline{s}$ of the same initial ideal with $\max(\overline{s})=\nu+1$, where if $k=1$, we may assume $\nu\geq 3$. To this end, write $d=\deg(s)$. If $s_\nu>1$, then by Lemma~\ref{lem:exponent_of_last_variable} we have 
\begin{equation}\label{eq:GFanProof_01}
    s_\nu=2d-k-(\nu-1)(m-1)\in \{2,\ldots,m-1\}.
\end{equation}
Now consider the monomial $t=(s/x_\nu)$. Using \eqref{eq:GFanProof_01}, we have
\begin{equation*}\label{eq:GFanProof_02}
        e_{\nu+1}:=2\deg(t)-k-(\nu)(m-1)
        =\ 2d-2-k-(\nu-1)(m-1)-(m-1)     \in\ \{ -m+1,\ldots,-2 \}.
\end{equation*}
Thus, $\bar{s}=t\cdot x_{\nu+1}^{|e_{\nu+1}|}$ is an $\mathbf{m}$-free monomial which is a minimal generator of the initial ideal, as
\begin{equation*}\label{eq:GFanProof_03}
        2\deg(\bar{s})-k-(\nu)(m-1)
        =\ 2\deg(t)+2|e_{\nu+1}|-k-(\nu)(m-1)
        =\ e_{\nu+1}+2|e_{\nu+1}|
        =\ |e_{\nu+1}|.
\end{equation*}
Finally, consider the case $s_\nu=1$. If $\nu=1$, this forces $k=1$, which is a case we have already handled. So assume $\nu>1$. Then $s_{\nu-1}>0$, as otherwise $s$ could not be a minimal generator. Indeed, otherwise consider the third-to-last node in the associated path. It must lie at least $1/2$ above the red line. Thus the end point of the whole path (which ends on a flat edge) is at least $(1/2)+1+(m-3)$ above the red line. The reflection of the last edge thus has slope at most $-(1+2+2m-6)=-2m+3$, which is not admissible. By a similar computation as above, we obtain a minimal generator $(s/x_{\nu-1})\cdot x_\nu x_{\nu+1}^{m-2}$ involving one additional variable. This completes the proof.
\end{proof}

\begin{remark}\label{rem:Cardinality_GFan_MixedCase}
   A result analogous to Proposition~\ref{prop:Cardinality_GFan_Equigenerated} does not hold when $P_{n,\mathbf{m}}$ is not equigenerated. For example, a Macaulay2 computation using the \texttt{gfaninterface} package shows that $I_{3,(2,3,4),2}$ has $5$ distinct reduced Gr\"{o}bner bases (as opposed to a divisor of $6 = 3!$, as established in Proposition~\ref{prop:Cardinality_GFan_Equigenerated} for the equigenerated case).
\end{remark}

\subsection{Enumerative properties of the Gr\"obner basis}
One interesting property of our Gröbner basis is that elements of a specified degree are only added at a unique number of variables. More precisely, if $\left\{ s \in \mathrm{Crit}_{n_1,\mathbf{m},k} \,\middle|\, \deg(s) = d \right\}$ and $\left\{ s \in \mathrm{Crit}_{n_2,\mathbf{m},k} \,\middle|\, \deg(s) = d \right\}$ are both non-empty, then 
$$ \left\{ s \in \mathrm{Crit}_{n_1,\mathbf{m},k} \,\middle|\, \deg(s) = d \right\}  =  \left\{ s \in \mathrm{Crit}_{n_2,\mathbf{m},k} \,\middle|\, \deg(s) = d \right\}, $$
when seen as elements of the same ring, see Proposition \ref{prop:GB_degree_element_count_sequences_general_case}. This leads us to study the properties of the sequence of non-negative integers $(g_{\mathbf{m},k}(d))_{d\geq k}$ defined by
\begin{equation}\label{eq:GB_degree_element_count_sequences}
    g_{\mathbf{m},k}(d) = \left| \left\{ s \in \mathrm{Crit}_{n,\mathbf{m},k} \,\middle|\, n \gg 0,\, \deg(s) = d \right\} \right|.
\end{equation}
As we will see later in Section \ref{sec:Applications}, several famous combinatorial sequences arise in this way such as the (generalized) Catalan, Motzkin and Riordan numbers.  Our first goal is to show that this sequence is well-defined, that is, for any $d \geq k$ and any sufficiently large $n_1,n_2$, we have
\[
\left| \left\{ s \in \mathrm{Crit}_{n_1,\mathbf{m},k} \,\middle|\, \deg(s) = d \right\} \right| = \left| \left\{ s \in \mathrm{Crit}_{n_2,\mathbf{m},k} \,\middle|\, \deg(s) = d \right\} \right|.
\]
In such comparisons, we always assume compatibility of the exponent vectors of pure variable powers. Specifically, when considering parameter values $(n_1,\mathbf{m},k)$ and $(n_2,\mathbf{m},k)$ with $n_1 < n_2$, we require that the vectors $\mathbf{m}$ agree on their first $n_1$ entries. Equivalently, we may regard $\mathbf{m} = (m_1, m_2, \ldots) \in \mathbb{Z}_{\geq 2}^{\mathbb{N}}$ as an infinite exponent sequence and consider its suitable finite truncations. 

\begin{definition}\label{def:Type_1_Type_2}
    Fix an exponent sequence $\mathbf{m}\in\mathbb{Z}_{\geq 2}^{\mathbb{N}}$ and a positive integer $k$. Given these parameters, for positive integers $n$ we define $\sigma_n=\max\{ m_1,\ldots, m_n \}$ and
    \[
    \tau_n=\left(\sum_{i=1}^n (m_i-1)\right) - (\sigma_n-1).
    \]
    We also set $\sigma_0=\tau_0=0$. Moreover, we say that $n$ is of \emph{Type 1} if $k\geq \sigma_{n-1}-\tau_{n-1}-1$ and that it is of \emph{Type 2} otherwise.
\end{definition}

Note that if $n\geq 3$ and $\mathbf{m}=(m,m,\dots)$ is a constant vector, then $n$ is of Type 1.

\begin{lemma}\label{lem:Non-empty_sets_Crit_n_and_their_degree_structure}
Let the parameters $\mathbf{m}, k$ be given, and for $j\geq 1$ write $R_j=\mathbf{k}[x_1,\ldots,x_j]$. Moreover, introduce the notation $D_j$ for the socle degree of $R_j/P_{j,\mathbf{m}}$ and $\delta_j$ for the socle degree of $R_j/I_{j,\mathbf{m},k}$. We extend this notation for $j=0$ via $D_0=\delta_0=0$. Further, let $s$ be an element of $\mathrm{Crit}_{n,\mathbf{m},k,n}$ with minimal possible $x_n$-degree and set $\deg(s)=d$. Then, for 
$n\geq 1$, we have
    \begin{enumerate}
        \item[{\rm (i)}] $\mathrm{Crit}_{n,\mathbf{m},k,n} \neq \emptyset$ only if $k+D_{n-1}-2\delta_{n-1}<m_n$.
        \item[{\rm (ii)}]  $s_n=k+D_{n-1}-2\delta_{n-1}$ and  $d=k+D_{n-1}-\delta_{n-1}$.
        \item[{\rm (iii)}] For every non-negative integer 
        \[
            e \leq \min\left\{ \frac{m_n - 1 - s_n}{2},\; \delta_{n-1} \right\},
        \]
        there exists an element $u \in \mathrm{Crit}_{n,\mathbf{m},k,n}$ of total degree $d + e$ such that $u_n = s_n + 2e$. The number of such elements in $\mathrm{Crit}_{n,\mathbf{m},k,n}$ is equal to $\HF(R_{n-1}/I_{n-1,\mathbf{m},k},\delta_{n-1} - e)$. \\
        Moreover, all elements of $\mathrm{Crit}_{n,\mathbf{m},k,n}$ have degree and $x_n$-degree of this form.
    \end{enumerate}
\end{lemma}
\begin{proof}
    To prove (i) and (ii), we first note that $k+D_{n-1}-2\delta_{n-1}>0$. Indeed, using the notation of Definition~\ref{def:Type_1_Type_2} and \cite[Theorem 1]{RRR} we see that
    \begin{equation}\label{eq:Socle_degrees_initial_ideals_cases}
        \delta_{n-1}=
        \begin{cases}
            \lfloor\frac{D_{n-1}+k-1}{2}\rfloor,& k\geq \sigma_{n-1}-\tau_{n-1}-1 \\
            \tau_{n-1}+k-1,& k< \sigma_{n-1}-\tau_{n-1}-1.
        \end{cases}
    \end{equation}
    In the first case, the relation $k+D_{n-1}-2\delta_{n-1}>0$ is obvious. In the second case, we have $\sigma_{n-1}-\tau_{n-1}>k+1$, and hence $k+D_{n-1}-2\delta_{n-1}=k+\tau_{n-1}+\sigma_{n-1}-1-2\tau_{n-1}-2k+2=\sigma_{n-1}-\tau_{n-1}-k+1>\sigma_{n-1}-\tau_{n-1}-\sigma_{n-1}+\tau_{n-1}+2=2>0$.

    By Lemma~\ref{lem:exponent_of_last_variable}, a monomial $s\in \mathrm{Crit}_{n,\mathbf{m},k,n}$ must satisfy $s_n=2d-k-D_{n-1}$. 
    Furthermore, we must have $s_{\leq n-1}\notin \mathrm{in}(I_{n-1,\mathbf{m},k})$. Note how increasing $s_n$ forces $\deg(s_{\leq n-1})$ to decrease. Thus the minimal possible $x_n$-degree for such a monomial $s$ is achieved when $\deg(s_{\leq n-1})=\delta_{n-1}$. In this case, $s_n=2d-k-D_{n-1}=2\delta_{n-1}+2s_n-k-D_{n-1}$, and hence $s_n=k+D_{n-1}-2\delta_{n-1}>0$. A suitable $\mathbf{m}$-free monomial with this degree structure exists only if $k+D_{n-1}-2\delta_{n-1}<m_n$. Its total degree is then $d=s_n+\delta_{n-1}=k+D_{n-1}-\delta_{n-1}$. Thus, (i) and (ii) hold.
    
    To prove (iii), first note that from the definition of $\mathrm{Crit}_{n,\mathbf{m},k,n}$, it follows that if two elements in the set differ in total degree by $\epsilon$, then their $x_n$-degrees differ by $2\epsilon$. Furthermore, every monomial in $\mathbf{k}[x_1,\ldots,x_{n-1}]$ not contained in $\ini(I_{n-1,\mathbf{m},k})$ can be extended to an element of $\mathrm{Crit}_{n,\mathbf{m},k,n}$ by multiplying with a suitable power of $x_n$ of degree at most $m_n-1$.
    Conversely, every element of $\mathrm{Crit}_{n,\mathbf{m},k,n}$ must be of this form. This correspondence implies the claimed degree pattern and counting formula, completing the proof.
\end{proof}

\begin{corollary}\label{cor:Crit_min_xn_degree_and_its_total_degree_type_1}
    Let the parameters $\mathbf{m}$ and $k$ be such that each $n > 1$ is of type 1. Consider an $n$ with $\mathrm{Crit}_{n,\mathbf{m},k,n} \neq \emptyset$. Let $s$ be an element of $\mathrm{Crit}_{n,\mathbf{m},k,n}$ with minimal possible $x_n$-degree. Then we have $s_n \in \{1,2\}$.
\end{corollary}

\begin{proof}
Using Lemma~\ref{lem:Non-empty_sets_Crit_n_and_their_degree_structure}(ii) and its notation, we have $s_n=k+D_{n-1}-2\delta_{n-1}$. Since $n$ is of type 1, \eqref{eq:Socle_degrees_initial_ideals_cases} gives $\delta_{n-1}=\left\lfloor \frac{D_{n-1}+k-1}{2}\right\rfloor$. We compute
\[
k+D_{n-1}-2\left\lfloor \frac{D_{n-1}+k-1}{2}\right\rfloor
=
\begin{cases}  
k+D_{n-1}-(D_{n-1}+k-1) & \text{if } D_{n-1} + k \text{ is odd}, \\
k+D_{n-1}-(D_{n-1}+k-2) & \text{if } D_{n-1} + k \text{ is even}.
\end{cases}
\]
In any case, $k+D_{n-1}-2\delta_{n-1}\in\{1,2\}$, which implies the statement.
\end{proof}

We now describe an efficient method for computing the sequences $g_{\mathbf{m},k}$, based on appropriate differences of Hilbert series coefficients for the Artinian monomial complete intersections $R/P_{n,\mathbf{m}}$.

\begin{proposition}\label{prop:GB_degree_element_count_sequences_general_case}
Let $d \geq k$ be the degree of a Gr\"obner basis element $g_s \in G_{n,\mathbf{m},k}$ for some $n \geq 1$, with $s = \ini(g_s) \in \mathrm{Crit}_{n,\mathbf{m},k,n}$. For integers $j\geq 0$, let $D_j$ and $\delta_j$ be defined as in Lemma~\ref{lem:Non-empty_sets_Crit_n_and_their_degree_structure}. Then the following hold.
    \begin{itemize}
        \item[{\rm (i)}] The integer $n=n(d,\mathbf{m},k)$ is uniquely determined by $d$, $\mathbf{m}$, and $k$. In particular, the sequences $g_{\mathbf{m},k}$ introduced in~\eqref{eq:GB_degree_element_count_sequences} are well-defined.

        \item[{\rm (ii)}] Let $\epsilon$ be the minimal degree of an element $s\in \mathrm{Crit}_{n,\mathbf{m},k,n}$ and set $\gamma=d-\epsilon$. Then the number of $\mathbf{m}$-free degree-$d$ elements in the Gr\"obner basis is given by
        \begin{equation*}\label{eq:GB_count_via_Hilbert_series_of_MCI_general}
            g_{\mathbf{m},k}(d) = \max\left\{ 0,\; \HF\left( \frac{\mathbf{k}[x_1,\ldots,x_{n-1}]}{P_{n-1,\mathbf{m}}}, \delta_{n-1} - \gamma \right)
            - \HF\left( \frac{\mathbf{k}[x_1,\ldots,x_{n-1}]}{P_{n-1,\mathbf{m}}}, \delta_{n-1} - \gamma - k \right) \right\}.
        \end{equation*}
    \end{itemize}
\end{proposition}
\begin{proof}
    To prove (i), introduce the notation $d_{\min}^{(n)}$ and $d_{\max}^{(n)}$ for the minimal and maximal total degree, respectively, of an element of $\mathrm{Crit}_{n,\mathbf{m},k,n}$, whenever this set is nonempty. It is enough to show that $d_{\max}^{(n_1)}<d_{\min}^{(n_2)}$ for $n_1<n_2$, whenever both numbers are defined. We use the notations of Definition~\ref{def:Type_1_Type_2}. Then by~\eqref{eq:Socle_degrees_initial_ideals_cases}, for $n$ of type 1 we have $\delta_{n-1}=\left\lfloor\frac{D_{n-1}+k-1}{2}\right\rfloor$ and for $n$ of type 2 we have $\delta_{n-1}=\tau_{n-1}+k-1$. Now, as a consequence of Lemma~\ref{lem:Non-empty_sets_Crit_n_and_their_degree_structure}(ii), we have $d_{\min}^{(n)}=k+D_{n-1}-\delta_{n-1}$, and hence
    \begin{equation*}\label{eq:min_degrees_crit_sets_cases}
        d_{\min}^{(n)}=
        \begin{cases}
            \left\lceil\frac{D_{n-1}+k+1}{2}\right\rceil, & \text{ if } n \text{ is of type 1},\\
            \sigma_{n-1},& \text{ if } n \text{ is of type 2}.
        \end{cases}
    \end{equation*}
    Indeed, the computation for $n$ of type 1 is easy. If $n$ is of type 2, then $k+D_{n-1}-\delta_{n-1}=k+\tau_{n-1}+\sigma_{n-1}-1-(\tau_{n-1}+k-1)=\sigma_{n-1}$.

    Now, by Lemma~\ref{lem:Non-empty_sets_Crit_n_and_their_degree_structure}(iii), for both types we have
    \begin{equation}\label{eq:max_degrees_crit_sets_cases}
    d_{\max}^{(n)}\leq d_{\min}^{(n)}+\min\left\{ \left\lfloor \frac{m_n-1 - \alpha}{2} \right\rfloor, \delta_{n-1} \right\}
    \end{equation}
    where $\alpha = D_{n-1}+k-2\delta_{n-1}$. We now show the desired inequality $d_{\max}^{(n_1)}<d_{\min}^{(n_2)}$ by means of a four-way case distinction.

    {\bf Case 1:} $n_1$ and $n_2$ are both of type 1. Then $d_{\min}^{(n_1)}=\left\lceil \frac{D_{n_1-1}+k+1}{2} \right\rceil$ and $d_{\min}^{(n_2)}=\left\lceil \frac{D_{n_2-1}+k+1}{2} \right\rceil$. We note that in this situation, $\alpha\in \{1,2\}$ satisfies $\alpha\equiv D_{n_1-1}+k\mod 2$. Thus we can use~\eqref{eq:max_degrees_crit_sets_cases} to estimate
    \begin{align*}
        d_{\max}^{(n_1)}
        &\leq
        d_{\min}^{(n_1)}+\left\lfloor \frac{m_{n_1}-1-\alpha}{2} \right\rfloor\\
        &=
        \begin{cases}
            d_{\min}^{(n_1)}+ \frac{m_{n_1}-1}{2} - 1, & m_{n_1}-1\text{ even} \\
            d_{\min}^{(n_1)}+ \frac{m_{n_1}-2}{2}=\left\lceil \frac{D_{n_1-1}+k+1+(m_{n_1}-2)}{2} \right\rceil, & m_{n_1}-1\text{ odd and }\alpha=1\\
            d_{\min}^{(n_1)}+ \frac{m_{n_1}-4}{2}=\left\lceil \frac{D_{n_1-1}+k+1+(m_{n_1}-2)}{2} \right\rceil -1, & m_{n_1}-1\text{ odd and }\alpha=2.
        \end{cases}\\
        &=\left\lceil \frac{D_{n_1-1}+k+1+(m_{n_1}-1)}{2} \right\rceil -1
    \end{align*}
    In all these cases, $d_{\max}^{(n_1)}\leq d_{\min}^{(n_2)}-1<d_{\min}^{(n_2)}$ follows.

    {\bf Case 2:} $n_1$ is of type 1 and $n_2$ is of type 2. Then $d_{\min}^{(n_1)}=\left\lceil \frac{D_{n_1-1}+k+1}{2} \right\rceil$ and $d_{\min}^{(n_2)}=\sigma_{n_2-1}$. Using~\eqref{eq:max_degrees_crit_sets_cases}, we estimate
    \begin{align*}
        d_{\max}^{(n_1)}
        \leq
        d_{\min}^{(n_1)}+\left\lfloor \frac{D_{n_1-1}+k-1}{2} \right\rfloor
        =
        D_{n_1-1}+k
        <
        D_{n_1-1}+\sigma_{n_2-1}-\tau_{n_2-1}-1
    \end{align*}
    where the strict inequality is implied by $n_2$ being of type 2. In order to estimate this further, note that $D_{n_1-1}=\sigma_{n_1-1}-1+\tau_{n_1-1}\leq \tau_{n_2-1}$. Indeed, the maximum $\sigma_{n_2-1}$ of $\{m_1,\ldots,m_{n_2-1}\}$ cannot be realized by $m_j$ with $j<n_1$, as otherwise we obtain a contradiction via the computation
    \begin{align*}
        k
        &<
        \sigma_{n_2-1}-\tau_{n_2-1}-1
        =
        \sigma_{n_1-1}-\tau_{n_2-1}-1\\
        &=
        \sigma_{n_1-1}-\tau_{n_1-1}-1+(\tau_{n_1-1}-\tau_{n_2-1})\\
        &\leq
        k+(\tau_{n_1-1}-\tau_{n_2-1})
        \leq
        k,
    \end{align*}
    where the penultimate inequality is due to $n_1$ being of type 1. Plugging this into the estimate of $d_{\max}^{(n_1)}$, we obtain $d_{\max}^{(n_1)}<D_{n_1-1}+\sigma_{n_2-1}-\tau_{n_2-1}-1\leq \sigma_{n_2-1}-1<d_{\min}^{(n_2)}$, as desired.

    {\bf Case 3:} $n_1$ is of type 2 and $n_2$ is of type 1. Then $d_{\min}^{(n_1)}=\sigma_{n_1-1}$ and $d_{\min}^{(n_2)}=\left\lceil \frac{D_{n_2-1}+k+1}{2} \right\rceil$. Since $\mathrm{Crit}_{n_1,\mathbf{m},k,n_1}$ is not empty, we have $m_{n_1}>k+D_{n_1-1}-2\tau_{n_1-1}-2k+2=\sigma_{n_1-1}-\tau_{n_1-1}-k+1$. The minimal possible value for $m_{n_1}$ is thus $\sigma_{n_1-1}-\tau_{n_1-1}-k+2$. For this value of $m_1$, using~\eqref{eq:max_degrees_crit_sets_cases}, we estimate
    \begin{align*}
        d_{\max}^{(n_1)}
        &\leq
        \sigma_{n_1-1}+\left\lfloor \frac{m_{n_1}-1-D_{n_1-1}-k+2k+2\tau_{n_1-1}-2}{2} \right\rfloor\\
        &=
        \sigma_{n_1-1}+\left\lfloor \frac{\sigma_{n_1-1}-\tau_{n_1-1}-k+1-D_{n_1-1}-k+2k+2\tau_{n_1-1}-2}{2} \right\rfloor\\
        &=
        \sigma_{n_1-1}+\left\lfloor \frac{0}{2} \right\rfloor
        =
        \sigma_{n_1-1}.
    \end{align*}
    Also for this value of $m_{n_1}$, the relation $D_{n_2-1}\geq D_{n_1-1}+(m_{n_1}-1)=\tau_{n_1-1}+(\sigma_{n_1-1}-1)+(m_{n_1}-1)$ gives
    \begin{align*}
        d_{\min}^{(n_2)}
        &=
        \left\lceil \frac{D_{n_2-1}+k+1}{2} \right\rceil
        \geq
        \left\lceil \frac{\tau_{n_1-1}+(\sigma_{n_1-1}-1)+(m_{n_1}-1)+k+1}{2} \right\rceil\\
        &=
        \left\lceil \frac{\tau_{n_1-1}+(\sigma_{n_1-1}-1)+(\sigma_{n_1-1}-\tau_{n_1-1}-k+1)+k+1}{2} \right\rceil\\
        &=
        \left\lceil \frac{2\sigma_{n_1-1}+1}{2} \right\rceil
        =
        \sigma_{n_1-1}+ \left\lceil \frac{1}{2} \right\rceil
        =
        \sigma_{n_1-1}+ 1
        >
        \sigma_{n_1-1}
    \end{align*}
    Now, using~\eqref{eq:max_degrees_crit_sets_cases} again, we see that to increase $d_{\max}^{(n_1)}$ by a positive integer $\epsilon$, we must increase $m_{n_1}$ and thus also $D_{n_2-1}$ by at least $2\epsilon$, and thereby also increase $d_{\min}^{(n_2)}$ by $\epsilon$. The desired inequality $d_{\max}^{(n_1)}<d_{\min}^{(n_2)}$ follows.

    {\bf Case 4:} $n_1$ and $n_2$ are both of type 2. Then $d_{\min}^{(n_1)}=\sigma_{n_1 -1}$ and $d_{\min}^{(n_2)}=\sigma_{n_2 -1}$. Moreover, $\delta_{n_1-1}=\tau_{n_1-1}+k-1$ and $\delta_{n_2-1}=\tau_{n_2-1}+k-1$. By~\eqref{eq:max_degrees_crit_sets_cases}, we obtain in particular
    \begin{align*}
    d_{\max}^{(n_1)}
    &\leq
    d_{\min}^{(n_1)}+\delta_{n_1-1}
    =
    \sigma_{n_1 -1}+\tau_{n_1-1}+k-1\\
    &<
    \sigma_{n_1 -1}+\tau_{n_1-1}+\sigma_{n_2 -1}-\tau_{n_2-1}-2\\
    &=
    D_{n_1-1}+1+\sigma_{n_2 -1}-\tau_{n_2-1}-2,
    \end{align*}
    where the strict inequality uses that $n_2$ is of type 2. We now show that $\sigma_{n_1-1}<\sigma_{n_2-1}$. Otherwise we have $\sigma_{n_1-1}=\sigma_{n_2-1}$. Then, since $\mathrm{Crit}_{n_1,\mathbf{m},k,n_1}$ is not empty, we have $m_{n_1}>k+D_{n_1-1}-2\tau_{n_1-1}-2k+2=\sigma_{n_1-1}-\tau_{n_1-1}-k+1$. This implies
    \begin{align*}
        k
        &>
        \sigma_{n_1-1}-\tau_{n_1-1}-m_{n_1}+1\\
        &=
        \sigma_{n_1-1}-1-\tau_{n_2-1}+\tau_{n_2-1}-\tau_{n_1-1}-m_{n_1}+2\\
        &>
        k+\tau_{n_2-1}-\tau_{n_1-1}-(m_{n_1}-1)+1
        \geq
        k+1,
    \end{align*}
    a contradiction. Hence indeed $\sigma_{n_1-1}<\sigma_{n_2-1}$. In turn, $D_{n_1-1}-\tau_{n_2-1}-1<0$, and thus we can complete our estimate $d_{\max}^{(n_1)}<\sigma_{n_2 -1}=d_{\min}^{(n_2)}$.

    To prove (ii), we apply Lemma~\ref{lem:Non-empty_sets_Crit_n_and_their_degree_structure}(iii) together with Proposition~\ref{prop:Counting_Critical_Paths} and Theorem~\ref{thm:Initial_Ideal_Equality}.
\end{proof}

\begin{example}\label{ex:Enumerative_sequence_(3,2,2,3)_2}
    Let $\mathbf{m}=(3,2,2,3,\ldots)$, and $k=2$. Then Example~\ref{ex:Reduced_GB_(3,2,2,3)_2} gives that $(g_{\mathbf{m},k})_{d\geq 2}=(1,1,3,\ldots)$, where the entries $g_{\mathbf{m},k}(d)$ depend also on $m_5,m_6,\ldots$ for $d>4$.
\end{example}

For illustrative examples of interesting sequences arising from $(g_{\mathbf{m},k}(d))_{d \geq k}$, see Section~\ref{sec:cubes}.

\begin{corollary}\label{cor:GB_degree_element_count_sequences_type_1}
    Let the parameters $\mathbf{m}$ and $k$ be such that each $n > 1$ is of type 1. Let $d \geq k$ be the degree of a Gr\"obner basis element $g_s \in G_{n,\mathbf{m},k}$ for some $n > 1$ with $s = \ini(g_s) \in \mathrm{Crit}_{n,\mathbf{m},k,n}$. Then the integer $n$ is uniquely determined by $d$, $\mathbf{m}$, and $k$, and is given by
        \begin{equation*}\label{eq:uniquely_defined_variable_count_for_given_degree}
            n = \max\left\{ \nu \in \mathbb{N} \;\middle\vert\; \frac{D_{\nu - 1} + k}{2} < d \right\}.
        \end{equation*}
\end{corollary}
\begin{proof}
We adopt the notation from Proposition~\ref{prop:GB_degree_element_count_sequences_general_case} and its proof. Let $n$ be the integer specified in the statement, it is uniquely determined by Proposition~\ref{prop:GB_degree_element_count_sequences_general_case}(i). Since $n$ is of type 1, we have $d_{\min}^{(n)} = \left\lceil \frac{D_{n-1} + k + 1}{2} \right\rceil$. Clearly, $d_{\min}^{(n)} \leq d \leq d_{\max}^{(n)}$. Moreover, since $\left\lceil \frac{D_{\nu - 1} + k + 1}{2} \right\rceil > \frac{D_{\nu - 1} + k}{2}$, this shows that $n \in N:= \left\{ \nu \in \mathbb{N} \;\middle\vert\; \frac{D_{\nu - 1} + k}{2} < d \right\}$.

The sequence $(D_\nu)_{\nu \geq 0}$ is monotonic and unbounded because $m_i \geq 2$ for all $i \geq 1$, so the set $N$ has a maximal element. Case 1 of the proof of Proposition~\ref{prop:GB_degree_element_count_sequences_general_case} further gives $d \leq d_{\max}^{(n)} \leq \left\lceil \frac{D_n + k - 1}{2} \right\rceil$, implying that $n + 1 \notin N$. Hence, we conclude that $n = \max N$.
\end{proof}

\begin{remark}\label{rem:largest_degree_of_GB_element}
For $j\geq 0$, let $\delta_j$ and $D_j$ be as in Lemma~\ref{lem:Non-empty_sets_Crit_n_and_their_degree_structure} and fix an integer $n\geq 1$. Let $q\in\{1,\ldots,n\}$ be the maximal index such that $m_q>k+D_{q-1}-2\delta_{q-1}$ and $\HF\left( \frac{\mathbf{k}[x_1,\ldots,x_{q-1}]}{P_{q-1,\mathbf{m}}}, \delta_{q-1} \right)- \HF\left( \frac{\mathbf{k}[x_1,\ldots,x_{q-1}]}{P_{q-1,\mathbf{m}}}, \delta_{q-1} - k \right)>0$. Then the maximal degree of an element in the Gr\"{o}bner basis $G_{n,\mathbf{m},k}$ with $\mathbf{m}$-free leading monomial is
\[
k+D_{q-1}-\delta_{q-1}+\min\left\{ \left\lfloor\frac{m_{q}-1-(k+D_{q-1}-2\delta_{q-1})}{2}\right\rfloor,\delta_{q-1} \right\}.
\]
For a constant sequence $\mathbf{m} = (m, m, \ldots)$ with $m \geq 2$, the maximal degree of an element in the Gr\"obner basis $G_{n,\mathbf{m},k}$, with $\mathbf{m}$-free leading monomial, is $k$ if $n = 1$ and $k < m$, and otherwise $\left\lceil (D_n + k - 1)/2 \right\rceil$. This follows from a computation in Case 1 of the proof of Proposition~\ref{prop:GB_degree_element_count_sequences_general_case}.
\end{remark}

\begin{example}\label{ex:largest_degree_of_GB_element_(2,3,2,20,3)_3}
 Let $n = 5$, $\mathbf{m} = (2, 3, 2, 20, 3)$, and $k = 3$. The largest index such that $m_q > k + D_{q-1} - 2\delta_{q-1}$ is $q = 4$.~The maximal degree of an element in 
 $G_{n,\mathbf{m},k}$~is
{\small{\begin{align*}
k+D_{3}-\delta_{3}+\min\left\{ \left\lfloor\frac{m_4-1-(k+D_{3}-2\delta_{3})}{2}\right\rfloor,\delta_{3} \right\}
=3+4-3+\min\left\{ \left\lfloor\frac{20-1-1}{2}\right\rfloor,3 \right\}=7.
\end{align*}}}

Indeed, we have a highest degree element $x_4^7 + 7x_4^6x_5 + 21x_4^5x_5^2$. Notice that in this case we have a complete intersection since $x_4^{20}\in (x_1^2,x_2^3,x_3^2,x_5^3,(x_1+\cdots + x_5)^3)$ in the spirit of Remark~\ref{rem:Almost_CI_is_CI}.
\end{example}

\section{Special cases and applications}\label{sec:Applications}
In this section, we focus on the equigenerated case where $\mathbf{m} = (m, m, \ldots, m)$ for some integer $m \geq 2$. 
For simplicity, when the context is clear, we write $I_{n,m,k}$ and $M_{n,m,k}$ in place of $I_{n,\mathbf{m},k}$ and $M_{n,\mathbf{m},k}$, respectively.
We begin by examining the cases $m = 2$ and $m = 3$, and show that the degree sequences of the reduced Gr\"obner bases in these cases are enumerated by convolutions of the Catalan and Motzkin numbers, respectively. In \S\ref{sec:quantum}, we consider the equigenerated case with $m \geq 2$ and $k = 1$, and examine the corresponding degree sequences, which also appear in quantum physics. In \S\ref{sub:n_even}, we establish a connection to $s$-Catalan numbers and the $(m-1)$-Catalan triangle. As a consequence, we prove in Corollary~\ref{cor:Log_convexity_rows_m-1_Catalan} that the rows of the $(m-1)$-Catalan triangle are log-concave.
In \S\ref{sub:WLP}, we investigate the Weak Lefschetz Property (WLP) for our ideals over fields of characteristic $p$. In particular, in Theorem~\ref{thm:Wlp_cond}, we provide a characterization of when $R/P_{n,\mathbf{m}}$ has the WLP in characteristic $p$, based on the initial ideal $\ini(I_{n,\mathbf{m},1})$.

\subsection{Equigenerated case $m = 2$}

The special case where the parameter vector is $\mathbf{m} = (2, \ldots, 2)$, for arbitrary values of $n$ and $k$, has been studied in \cite{kling2024grobnerbasesresolutionslefschetz}, where an alternative type of lattice path was introduced. The lattice paths and red lines defined in the present work correspond to those described in \cite{kling2024grobnerbasesresolutionslefschetz} under the invertible linear transformation  
$$\Phi: \mathbb{R}^2\to \mathbb{R}^2,\; \begin{bmatrix}
    x\\y
\end{bmatrix}\mapsto  \begin{bmatrix}
    0 & 1 \\ 1 & -1
\end{bmatrix}\begin{bmatrix}
    x\\y
\end{bmatrix}\, .$$

Figure~\ref{fig:Squares_Lattice_Paths_Transformation} illustrates this transformation with an example.

\begin{figure}[!ht]
    \centering
    \begin{tikzpicture}
        \draw[help lines] (0,-1) grid (4,2);
        \draw[line width=0.02cm, color=red] (0,-1)--(4,1);
        \draw[line width=0.1cm] (0,0)--(1,0)--(2,1)--(3,1)--(4,1);
        \draw[left] (0,0) node(origin){$(0,0)$};
        \draw[left] (0,-1) node(red){$(0,-1)$};
        \draw[right] (4,1) node(path){$(4,1)$};
    \end{tikzpicture}
    \begin{tikzpicture}
        \draw[help lines] (0,0) grid (1,3);
        \draw[line width=0.02cm, color=red] (0,2)--(1,3);
        \draw[line width=0.1cm] (0,0)--(0,1)--(1,1)--(1,2)--(1,3);
        \draw[left] (0,0) node(origin){$(0,0)$};
        \draw[left] (0,2) node(red){$(0,2)$};
        \draw[right] (1,3) node(path){$(1,3)$};
    \end{tikzpicture}
\caption{(Left) The lattice path corresponding to the leading term $x_1 x_3 x_4 \in \mathrm{in}(I_{4,(2,2,2,2),4})$, shown together with the graph of $L_{4,(2,2,2,2),4}$; (Right) The image of this configuration under the transformation $\Phi$, as introduced in \cite{kling2024grobnerbasesresolutionslefschetz}.}
    \label{fig:Squares_Lattice_Paths_Transformation}
\end{figure}
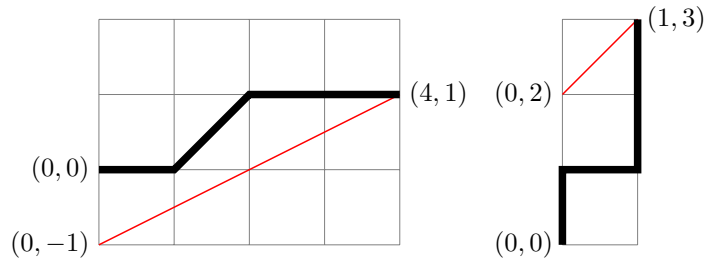

One can check that this transformation is a bijection between paths corresponding to leading terms in the Gr\"obner basis and that the formula \eqref{eq:more_explicit_gb_polynomial_g_s} for the Gr\"obner basis elements does specialize to the elementary symmetric polynomials way of expressing them from \cite{kling2024grobnerbasesresolutionslefschetz}. By \cite[Theorem 2.21]{kling2024grobnerbasesresolutionslefschetz}, for $\mathbf{m}=(2,2,\ldots)$ the sequence $g_{\mathbf{m},k}$ is a shift of the $(k-1)$-fold self convolution of the sequence of Catalan numbers.

\subsection{Equigenerated case $m = 3$}\label{sec:cubes}

Given a fixed parameter value $k$, the sequence $g_{(3,\ldots,3),k}$, counting the elements in the reduced Gr\"obner basis of $I_{n,3,k}$ whose leading terms are divisible by $x_n$ are enumerated by convolutions of the Motzkin numbers and a related sequence known as the Riordan numbers. To state this precisely, we first recall the definitions of these sequences.

\begin{definition}\label{def:Motzkin_Riordan}
A \emph{Motzkin path} of length $n$ is a lattice path from $(0,0)$ to $(n,0)$ consisting of steps that never fall below the $x$-axis, where each step is one of the following: an \emph{up step} $(1,1)$, a \emph{flat step} $(1,0)$, or a \emph{down step} $(1,-1)$.
Using these paths, two integer sequences are defined.
\begin{itemize}
    \item The \emph{$n$th Motzkin number} $M(n)$ counts the number of Motzkin paths of length $n$ \cite[A001006]{OEIS_2018}. By convention, $M(0) = 1$, corresponding to the empty path.
    \item The \emph{$n$th Riordan number} $R(n)$ counts the number of Motzkin paths of length $n$ with no flat steps on the $x$-axis \cite[A005043]{OEIS_2018}. By convention, $R(0) = 1$ and $R(1) = 0$, since the only Motzkin path of length one is a flat step on the $x$-axis. \footnote{See also \cite{CatalanMotzkinRiordan} for a detailed exploration of the relationships between Catalan, Motzkin, and Riordan numbers. Notably, the identity $M(n) = R(n) + R(n+1)$ relates the Motzkin and Riordan sequences.}
\end{itemize}
\end{definition}
Table~\ref{tab:MotzkinRiordan} lists the initial values of these sequences.

\begin{table}[h!]
\centering
    $$\begin{array}{|l||r|r|r|r|r|r|r|r|r|r|r|}
        \hline
        n & 0 & 1 & 2 & 3 & 4 & 5 & 6 & 7 & 8 \\
        \hline
        \hline
        M(n) & 1 & 1 & 2 & 4 & 9 & 21 & 51 & 127 & 323 \\
        \hline
        R(n) & 1 & 0 & 1 & 1 & 3 & 6 & 15 & 36 & 91 \\
        \hline
    \end{array}$$
    \caption{Motzkin and Riordan numbers.}
    \label{tab:MotzkinRiordan}
\end{table}

We now define a family of integer sequences that count specific elements in the Gr\"obner basis. 
\begin{definition}\label{def:GB_elems_sequences_m=3}
    For each integer $k\geq 1$, let $\delta_k = \lceil \frac{k}{2} \rceil$. We define the integer sequence 
    $g_{3,k} \colon \mathbb{Z}_{\geq 0} \to \mathbb{Z}_{\geq 0}$ by
    \begin{equation}\label{eqn:GB_elems_sequences_m=3}
        g_{3,k}(n) := \left|  \mathrm{Crit}_{n+\delta_k, 3, k, n+\delta_k}  \right|.
    \end{equation}
\end{definition}

The sequence $g_{3,k}$ is simply a shifted version of the sequence $g_{(3,\ldots,3),k}$ from~\eqref{eq:GB_degree_element_count_sequences} where the shift is such that $g_{3,k}(0)$ equals the first nonzero value of $g_{(3,\ldots,3),k}$. Equivalently, $g_{3,k}$ is obtained by removing the leading zeros from $g_{(3,\ldots,3),k}$.
\begin{remark}
Recall that each set $\mathrm{Crit}_{n,3,k,j}$ can contain elements of only a single degree, since the degree is determined by the $x_j$-degree, and different $x_j$-degrees must differ by even numbers. However, the set $\{1,2\}$ of allowed $x_j$-degrees is too small to accommodate distinct $x_j$-degrees satisfying this condition. Moreover, for $\mathbf{m} = (3,\ldots,3)$, we are always in Case 1 in the terminology of the proof of Proposition~\ref{prop:GB_degree_element_count_sequences_general_case}(i). A careful analysis of the proof shows that $d_{\max}^{(n_1)} + 1 = d_{\min}^{(n_1+1)}$ holds in all cases, except when $k = 1$ and $n_1 = 1$, in which case $\mathrm{Crit}_{2,3,1,2} = \emptyset$.
\end{remark}

Table~\ref{tab:Number_of_GB_elems_m=3_various_k} provides the first few values of the sequences $g_{3,1}$, \ldots, $g_{3,6}$.

\begin{table}[h!]
\centering
    $$\begin{array}{|l||r|r|r|r|r|r|r|r|r|r|r|}
        \hline
        n & 0 & 1 & 2 & 3 & 4 & 5 & 6 & 7 & 8 \\
        \hline
        \hline
        g_{3,1}(n) & 1 & 0 & 1 & 1 & 3 & 6 & 15 & 36 & 91\\
        \hline
        g_{3,2}(n) & 1 & 1 & 2 & 4 & 9 & 21 & 51 & 127 & 323\\
        \hline
        g_{3,3}(n) & 1 & 1 & 3 & 6 & 15 & 36 & 91 & 232 & 603\\
        \hline
        g_{3,4}(n) & 1 & 2 & 5 & 12 & 30 & 76 & 196 & 512 & 1353 \\
        \hline
        g_{3,5}(n) & 1 & 2 & 6 & 15 & 40 & 105 & 280 & 750 & 2025 \\
        \hline
        g_{3,6}(n) & 1 & 3 & 9 & 25 & 69 & 189 & 518 & 1422 & 3915 \\
        \hline
    \end{array}$$
    \caption{Numbers of cube-free Gr\"{o}bner basis elements of $I_{n+\delta_k,3,k}$ for $1\leq k \leq 6$.}
    \label{tab:Number_of_GB_elems_m=3_various_k}
\end{table}
The sequences $g_{3,1}$ and $g_{3,2}$ coincide with the Riordan and Motzkin numbers, respectively, for the initial values of $n$ as seen in Tables~\ref{tab:MotzkinRiordan} and \ref{tab:Number_of_GB_elems_m=3_various_k}. This is not a coincidence. It follows naturally from the lattice path interpretation of the leading terms in the Gr\"obner basis. In fact, these two sequences serve as building blocks, all sequences $g_{3,k}$ can be obtained from them via convolution. Before proving this result, we briefly recall the necessary definitions.

\begin{definition}\label{def:GeneratingFunctions_Convolutions}
Let $S, S_1, S_2 : \mathbb{Z}_{\geq 0} \to \mathbb{Z}_{\geq 0}$ be integer sequences. We define
\begin{itemize}
    \item The \emph{generating function} of $S$ as the formal power series
       $ \mathrm{GF}(S) = \sum_{n \geq 0} S(n) y^n$.
    
    \item The \emph{convolution} $S_1 \ast S_2$ as the integer sequence whose generating function satisfies
     $\mathrm{GF}(S_1 \ast S_2) = \mathrm{GF}(S_1) \cdot \mathrm{GF}(S_2)$.
\end{itemize}
\end{definition}

\begin{theorem}\label{thm:convolutions_sequences_m=3}
Let $k \geq 1$, and write $k = 2 q_k + r_k$ with integers $q_k \geq 0$ and $r_k \in \{0,1\}$. Then the generating function of the sequence $g_{3,k}$ defined in \eqref{eqn:GB_elems_sequences_m=3} satisfies
\begin{equation}\label{eq:Convolution}
    \mathrm{GF}(g_{3,k}) = \bigl(\mathrm{GF}(M)\bigr)^{q_k} \cdot \bigl(\mathrm{GF}(R)\bigr)^{r_k}.
\end{equation}
where $M$ and $R$ denote the Motzkin and Riordan number sequences, respectively.
\end{theorem}
\begin{proof}
    We start by analyzing the case $k=1$. Here, $\delta_1 = \lceil \frac{1}{2} \rceil = 1$. For each $n \geq 0$, the term $g_{3,1}(n)$ counts the elements of $M_{n+1,3,1}$ that are divisible by $x_{n+1}$. These correspond precisely to Motzkin paths of length $n+1$ ending at $(n+1, 0)$ whose last step is a flat step on the $x$-axis and which have no flat steps on the $x$-axis prior to the last step.
By deleting this last flat step, we obtain a bijection between these paths and Motzkin paths of length $n$ with no flat steps on the $x$-axis. The number of such paths is exactly $R(n)$.
Therefore, we conclude that $g_{3,1} = R$, and its generating function is given by
$\mathrm{GF}(g_{3,1}) = \mathrm{GF}(M)^0 \cdot \mathrm{GF}(R)^1$. 
(Note that this matches the decomposition $1 = 2 \cdot 0 + 1$, so $q_1 = 0$ and $r_1 = 1$.)

Next, we consider the case $k=2$. Here, $\delta_2 = \lceil \frac{2}{2} \rceil = 1$. For each $n \geq 0$, $g_{3,2}(n)$ counts the elements of $M_{n+1, 3, 2}$ divisible by $x_{n+1}$. These correspond to lattice paths ending at the point $(n+1, -1)$ whose last step is a down step, and which do not touch the horizontal line $y = -1$ at any earlier point.
By removing this last down step, we establish a bijection with the set of Motzkin paths of length $n$. The number of such paths is precisely $M(n)$.
Hence, we conclude that $g_{3,2} = M$, whose generating function satisfies
    $\mathrm{GF}(g_{3,2}) = \mathrm{GF}(M)^1 \cdot \mathrm{GF}(R)^0$. (Note that $2 = 2 \cdot 1 + 0$, so $q_2 = 1$ and $r_2 = 0$.)

 We now consider an arbitrary even value $k = 2 q_k$, so that $\delta_k = q_k$. The sequence $g_{3,k}(n)$ counts the elements of $M_{n + q_k, 3, k}$ divisible by $x_{n + q_k}$. These correspond to lattice paths ending at the point $(n + q_k, -q_k)$ whose last step is a down step, and which do not touch the horizontal line $y = -q_k$ at any earlier step.
Such a path $P$ can be decomposed into $q_k$ translated lattice paths, each associated with elements of $M_{n_1 + 1, 3, 2}$, \ldots, $M_{n_{q_k} + 1, 3, 2}$ divisible by $x_{n_1 + 1}$, \ldots, $x_{n_{q_k} + 1}$ respectively, where the integers $n_1, \ldots, n_{q_k} \geq 0$ satisfy $n_1 + \cdots + n_{q_k} = n$.
More precisely, the first segment $P_1$ is the initial portion of $P$ up to its first intersection with the line $y = -1$. The second segment is the translation of the path following $P_1$ up to the first intersection with $y = -2$, and so forth, until the $q_k$th segment. Thus,
\begin{equation}\label{eqn:Convolution_computation_even_case_m=3}
    g_{3,k}(n + q_k) = \sum_{\substack{(n_1, \ldots, n_{q_k}) \in \mathbb{Z}_{\geq 0}^{q_k} \\ n_1 + \cdots + n_{q_k} = n}} \prod_{j=1}^{q_k} g_{3,2}(n_j),
\end{equation}
which implies that the generating function satisfies \eqref{eq:Convolution}, 
as claimed.

The case of an arbitrary odd value $k = 2 q_k + 1$, where $\delta_k = q_k + 1$ and $r_k = 1$, is similar to the even case. Here, $g_{3,k}$ counts elements of $M_{n+\delta_k, 3, k}$ divisible by $x_{n+\delta_k}$. These correspond to lattice paths ending at $(n + q_k, -q_k)$ whose last edge is flat and have no flat edge on the horizontal line $y = -q_k$ prior to this final edge.
As before, we decompose such a path into translations of simpler lattice paths. The first $q_k$ segments $P_1, \ldots, P_{q_k}$ are constructed exactly as in the even case $k = 2 q_k$. The remaining segment is a translation of a Motzkin path ending with a flat edge that has no flat edges on its starting level before this final edge.
The total number of steps in these $\delta_k$ segments sums to $n + \delta_k$, the total step count of the original path $P$. An analogous calculation to \eqref{eqn:Convolution_computation_even_case_m=3} then shows that
$\mathrm{GF}(g_{3,k})=\mathrm{GF}(g_{3,2})^{q_k}\cdot \mathrm{GF}(g_{3,1})=\mathrm{GF}(M)^{q_k}\cdot \mathrm{GF}(R)$.
\end{proof}

\begin{example}\label{ex:GB_path_decomposition_m=3}
We illustrate Theorem~\ref{thm:convolutions_sequences_m=3} with two examples in Figures~\ref{fig:Path_decomposition_m=3_odd_case} and~\ref{fig:Path_decomposition_m=3_even_case}, showing how lattice paths for Gr\"{o}bner basis leading terms decompose.

  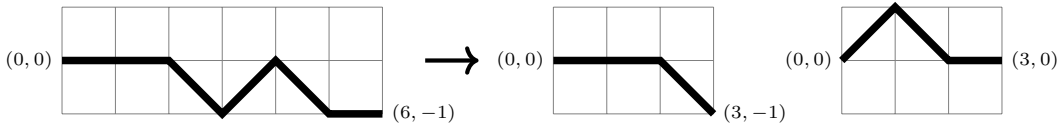
\begin{figure}[!ht]
\centering
\begin{tikzpicture}[scale=0.71]
    \begin{scope}[shift={(0,0)}]
        \draw[help lines] (0,-1) grid (6,1);
        \draw[line width=0.1cm] (0,0)--(1,0)--(2,0)--(3,-1)--(4,0)--(5,-1)--(6,-1);
        \draw[left] (0,0) node(origin){\scriptsize$(0,0)$};
        \draw[right] (6,-1) node(path){\scriptsize$(6,-1)$};
    \end{scope}

 \draw[->, ultra thick] (6.8,0) -- (7.8,0);

    \begin{scope}[shift={(9.2,0)}]
        \draw[help lines] (0,-1) grid (3,1);
        \draw[line width=0.1cm] (0,0)--(1,0)--(2,0)--(3,-1);
        \draw[left] (0,0) node(origin){\scriptsize$(0,0)$};
        \draw[right] (3,-1) node(path){\scriptsize$(3,-1)$};
    \end{scope}

    \begin{scope}[shift={(14.6,0)}]
        \draw[help lines] (0,-1) grid (3,1);
        \draw[line width=0.1cm] (0,0)--(1,1)--(2,0)--(3,0);
        \draw[left] (0,0) node(origin){\scriptsize$(0,0)$};
        \draw[right] (3,0) node(path){\scriptsize$(3,0)$};
    \end{scope}
\end{tikzpicture}
\caption{The lattice path corresponding to the Gr\"{o}bner basis leading term $x_1 x_2 x_3^2 x_5^2 x_6 \in \mathrm{in}(I_{6,3,3})$, decomposed into translations of lattice paths associated with the leading terms $x_1 x_2 x_3^2 \in \mathrm{in}(I_{3,3,2})$ and $x_2^2 x_3 \in \mathrm{in}(I_{3,3,1})$.}
\label{fig:Path_decomposition_m=3_odd_case}
\end{figure}

   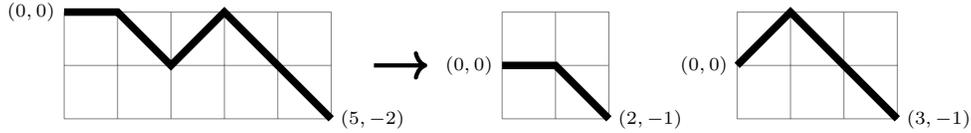
\begin{figure}[!ht]
    \centering
    \begin{tikzpicture}[scale=0.71]
        \begin{scope}[shift={(0,0)}]
            \draw[help lines] (0,-2) grid (5,0);
            \draw[line width=0.1cm] (0,0)--(1,0)--(2,-1)--(3,0)--(4,-1)--(5,-2);
            \draw[left] (0,0) node(origin){\scriptsize$(0,0)$};
            \draw[right] (5,-2) node(path){\scriptsize$(5,-2)$};
        \end{scope}

        \draw[->, ultra thick] (5.8,-1) -- (6.8,-1);

        \begin{scope}[shift={(8.2,-1)}]
            \draw[help lines] (0,-1) grid (2,1);
            \draw[line width=0.1cm] (0,0)--(1,0)--(2,-1);
            \draw[left] (0,0) node(origin){\scriptsize$(0,0)$};
            \draw[right] (2,-1) node(path){\scriptsize$(2,-1)$};
        \end{scope}

        \begin{scope}[shift={(12.6,-1)}]
            \draw[help lines] (0,-1) grid (3,1);
            \draw[line width=0.1cm] (0,0)--(1,1)--(2,0)--(3,-1);
            \draw[left] (0,0) node(origin){\scriptsize$(0,0)$};
            \draw[right] (3,-1) node(path){\scriptsize$(3,-1)$};
        \end{scope}
    \end{tikzpicture}
    \caption{The lattice path associated with the Gr\"{o}bner basis leading term $x_1 x_2^2 x_4^2 x_5^2 \in \mathrm{in}(I_{5,3,4})$, decomposed into translations of lattice paths associated with the leading terms $x_1 x_2^2\in \mathrm{in}(I_{2,3,2})$ and $x_2^2 x_3^2\in \mathrm{in}(I_{3,3,2})$.}
    \label{fig:Path_decomposition_m=3_even_case}
\end{figure}

\end{example}

\subsection{Equigenerated case with $k=1$ in quantum physics}\label{sec:quantum}
In this subsection, we consider the case $\mathbf{m} = (m,\ldots,m)$ with $m \geq 2$ and $k = 1$. We describe an application of the integer sequence $g_{\mathbf{m},1}$, which counts the number of $\mathbf{m}$-free elements in the reduced Gr\"obner basis of $I_{n,\mathbf{m},1}$ by degree for sufficiently large $n$, as introduced in~\eqref{eq:GB_degree_element_count_sequences}. It turns out that this sequence arises in 
quantum physics~\cite{Cohen_EtAl_Entanglement, curtright2017spin}.

More precisely, in~\cite[Section 1]{Cohen_EtAl_Entanglement}, the authors introduce an operator $W$ as an \emph{entanglement witness}, i.e., a tool for detecting whether a quantum state is entangled. The operator has eigenvalues $W_j = j(j+1)$, indexed by $j \in \{n/2 \mid j \in \mathbb{Z}_{\geq 0}\}$. They further develop a lattice path method to compute the degeneracy of each eigenvalue, which corresponds to the number of distinct spin-$j$ states associated with $W_j$.
The allowed lattice paths depend on the spin $\sigma \in (1/2)\cdot \mathbb{Z}_{\geq 0}$ of the quantum system. The degeneracy of states with eigenvalue $W_j$ is given by the number of lattice paths starting at $(0,0)$ and ending at $(N,j)$, for a prescribed positive integer $N$, corresponding to the number of particles in the system, using admissible steps of the form $(1,\sigma), (1,\sigma-1), \ldots, (1,-\sigma)$. These paths must satisfy the additional constraint
\begin{equation}\label{eq:Quantum_constraint}
    |j_1 - \sigma| \leq j_2 \leq j_1 + \sigma
\end{equation}
for each step from $(x,j_1)$ to $(x+1,j_2)$~as detailed in \cite[Eqs. (13)-(15)]{Cohen_EtAl_Entanglement}.
Among these, the states with eigenvalue $W_0$ corresponding to lattice paths ending at $(N,0)$, are of particular interest, as they are known to be decoherence-free~\cite{Cohen_EtAl_Entanglement}. Hence, we focus exclusively on these special states. The associated sequences are called \emph{spin $s$-Catalan numbers} in~\cite[Definition 3]{Linz_s_Catalan}.

As a first observation, note that the number of allowed steps is $2\sigma + 1$, which is \emph{odd} when $\sigma \in \mathbb{Z}$ and \emph{even}, namely $2\lceil \sigma \rceil \geq 2$, when $\sigma \notin \mathbb{Z}$. In this way, the spin values $\sigma$ effectively parameterize the integers $m \geq 2$. As in Section~\ref{sec:InitialIdeal}, the lattice paths starting at $(0,0)$, ending at some point $(N,j)$, and composed of these allowed steps (ignoring additional constraints) are in bijection with the $\textbf{m}$-free terms in $\mathbf{k}[x_1,\ldots,x_N]$.

If we associate the slope $\sigma$ with exponent $0$, the slope $\sigma - 1$ with exponent $1$, and so on down to the slope $-\sigma$ with exponent $m - 1$, then we can again visualize the triangular scheme of Hilbert series coefficients described in Remark~\ref{rem:Triangular_scheme}, using this new convention. Under this setup, the symmetry points of the columns lie on the $x$-axis, allowing us to apply the same analysis of the initial ideal of $I_{n,\mathbf{m},1}$ as in Section~\ref{sec:InitialIdeal}, but with the red boundary line now given by $y = -\tfrac{1}{2}$.
As a result, we can now characterize the lattice paths corresponding to the states with eigenvalue~$W_0$.

\begin{lemma}\label{lem:Characterization_of_decoherence_free_paths}
Let $m=2\sigma+1$. Consider the set of lattice paths starting at $(0,0)$ and ending at $(N,0)$, whose steps are taken from the set 
$\{(1,\sigma), (1,\sigma-1),\ldots,(1,-\sigma)\}$
that never pass below the $x$-axis and satisfy the constraint~\eqref{eq:Quantum_constraint}.  
These paths are in bijection with the set of $\mathbf{m}$-free monomials of degree $\sigma N$ in the quotient ring 
$\mathbf{k}[x_1,\ldots,x_N]/\ini(I_{N,\mathbf{m},1})$.
\end{lemma}

\begin{proof}
We need to show that the paths described in the lemma are exactly the non-critical paths with respect to reflection on the line $y = -\frac{1}{2}$. By definition, these are the paths that never touch or cross $y = -\frac{1}{2}$ and have no edge $AB$ for which the reflected edge $AB'$ has an admissible slope.
Since the allowed slopes are of the form $\tau/2$ for integers $\tau$, the condition that a path never touches or crosses $y = -\frac{1}{2}$ is equivalent to the path never going below the $x$-axis.
For the second condition, let $A = (x, j_1)$ and $B = (x+1, j_2)$ be an edge of the path. Reflecting this edge across 
$y = -\frac{1}{2}$ maps $B$ to $B' = (x+1, -j_2 - 1)$. The slope of $AB'$ is then $-(j_1 + j_2 + 1)$.

Suppose, for contradiction, that this slope is admissible, i.e., $-(j_1 + j_2 + 1) \geq -\sigma$, which implies $j_1 + j_2 + 1 \leq \sigma$.
From the constraint~\eqref{eq:Quantum_constraint}, in particular $\sigma - j_1 \leq j_2$, we get $-j_1 \leq j_2 - \sigma$. Adding that to $j_1 + j_2 + 1 \leq \sigma$ gives $j_2 + 1 \leq j_2$, a contradiction.
Hence, the slope of $AB'$ cannot be admissible, and the lattice path is non-critical.
The converse inclusion that every non-critical path satisfies the assumptions of the lemma, is proved similarly.

The statement on the degrees follows by an argument analogous to the proof of Lemma~\ref{lem:path_set_up_mixed_m}.
\end{proof}

Before proceeding to count the degeneracy of states with eigenvalue $W_0$ using the sequences $g_{\mathbf{m},1}$, we extend our notation to accommodate half-integer arguments by defining
\[
g_{\mathbf{m},k}(z/2) = 0 \quad \text{for all odd integers } z.
\]

\begin{proposition}\label{prop:Counting_degeneracy_for_j=0}
Let $\sigma$ be the spin for a system of $N$ particles and set $m = 2\sigma + 1$ with $\mathbf{m} = (m, \ldots, m)$. Then the degeneracy of states with eigenvalue $W_0$, that is, the number of distinct states with spin $0$, is given by 
$g_{\mathbf{m},1}(\sigma N + 1)$.
\end{proposition}

\begin{proof}
Since the number of allowed slopes for spin $\sigma$ is $m = 2\sigma + 1$, this follows from observing that from any path described in the statement of Lemma~\ref{lem:Characterization_of_decoherence_free_paths}, we obtain the path of a minimal $\mathbf{m}$-free generator of $\ini(I_{N+1,\mathbf{m},1})$ by appending a step with slope $\sigma - 1$ at the end (this is equivalent to multiplying by the variable $x_{N+1}$). Indeed, this last step goes from $A = (N, 0)$ to $B = (N+1, \sigma - 1)$. We have $B' = (N+1, -(\sigma - 1) - 1) = (N+1, -\sigma)$, and the slope of $AB'$ is $-\sigma$, which is the minimal admissible slope. Applying Lemma~\ref{lem:exponent_of_last_variable}, Theorem~\ref{thm:Gen_sys_for_critical_path_ideal}, and Lemma~\ref{lem:Characterization_of_decoherence_free_paths} yields the claim.
\end{proof}

\subsection{Equigenerated case with $n$ even and $k=1$ applied to $s$-Catalan numbers}\label{sub:n_even}

In this section, we apply the enumerative properties of the Gr\"{o}bner bases $G_{2n,\mathbf{m},1}$ to the study of generalized Catalan numbers, denoted $s$-Catalan numbers. The $1$-Catalan numbers coincide with the usual Catalan numbers, and by \cite[Theorem 2.21]{kling2024grobnerbasesresolutionslefschetz}, these coincide with $g_{(2,2,\ldots),1}$. Here, we will relate the $s$-Catalan numbers for any $s\geq 1$ with the sequences $g_{(s+1,s+1,\ldots,),1}$. We start by recalling the definition of $s$-binomial coefficients.

The \emph{$s$-binomial coefficients} ${\binom{n}{d}\!}_s$ are defined (e.g. \cite{ghemit2024two,Linz_s_Catalan}) for positive integers $n$ and $s$ via 
$$(1+x+\cdots +x^s)=\sum_{d\in\mathbb{Z}} {\binom{n}{d}\!\!}_s x^d.$$
In our setting, we adopt the notation $m = s$, so the definition becomes
\begin{equation}\label{eq:bi_s_nomial_coeffs_via_HS_of_MCI}
   {\binom{n}{d}\!}_{m-1} = \HF(R/P_{n,\mathbf{m}},d)\,,
\end{equation}
meaning that the nonzero ``$(m-1)$-binomial'' coefficients correspond to the coefficients of the Hilbert series of the Artinian monomial complete intersection defined by an ideal equigenerated in degree $m \geq 2$. This connection allows us to relate our results to the $s$-Catalan numbers. (See, e.g.~\cite{Linz_s_Catalan} 
for the definition of $s$-Catalan numbers and~\cite{ghemit2024two} for generalized Catalan triangles.)

\begin{definition}
\label{def:s_Catalan_numbers_and_triangle}
The \emph{$s$-Catalan numbers} for a positive integer $s$ are defined as
\begin{equation}\label{eq:s_Catalan_numbers}
    C_{n}^{(s)}={\binom{2n}{sn}\!\!}_s-{\binom{2n}{sn+1}\!\!}_s.
\end{equation}
The coefficients of the \emph{$s$-Catalan triangle} are given by $C_{0,0}^{(s)}=1$ and
\begin{equation}\label{eq:s_Catalan_triangle}
    C_{n,k}^{(s)}={\binom{2n}{sn+k}\!\!}_s-{\binom{2n}{sn+k+1}\!\!}_s\,
\end{equation}
for $n>0$ and $k\geq 0$. The $s$-Catalan triangle, which is written with entry $C_{n,k}^{(s)}$ in row $n$ and column $k$, contains the $s$-Catalan numbers in its first column.
\end{definition}

\begin{remark}\label{rem:s_Catalan_vs_m-1_Catalan}
    Throughout the remainder of this work, we will consistently write $m - 1$ in place of $s$ to align with our established notation, and accordingly, write \emph{$(m-1)$-Catalan numbers}.
\end{remark}

In~\cite{ghemit2024two} it was proven that the rows of the $2$-Catalan triangle are log-concave. 
Using our results, we provide the following lattice path interpretation of the numbers in the $(m-1)$-Catalan triangle for any positive integer $m-1$.

\begin{proposition}\label{prop:Combinatorial_interpretation_s_Catalan_triangle}
    For $m\geq 2$, $n\geq 1$, and $k\geq 0$, the coefficient $C_{n,k}^{(m-1)}$ in the $(m-1)$-Catalan triangle counts the $m$-admissible lattice paths of degree $n(m-1)-k$ whose corresponding $\mathbf{m}$-free monomials are not in the initial ideal $\ini(I_{2n,\mathbf{m},1})$.
\end{proposition}
\begin{proof}
Using  \eqref{eq:bi_s_nomial_coeffs_via_HS_of_MCI} and \eqref{eq:s_Catalan_triangle}, along with the properties of the Hilbert series of Artinian monomial complete intersections stated in Remark~\ref{rem:Properties_Artinian_monomial_CIs}, Proposition~\ref{prop:Counting_Critical_Paths} and Theorem~\ref{thm:Initial_Ideal_Equality}, we compute
    \begin{align*}
        C_{n,k}^{(m-1)}
        =&
        \binom{2n}{(m-1)n+k}_{m-1}-\binom{2n}{(m-1)n+k+1}_{m-1}\\
        =&
        \HF(R/P_{2n,\mathbf{m}},(m-1)n+k)-\HF(R/P_{2n,\mathbf{m}},(m-1)n+k+1)\\
        =&
        \HF(R/P_{2n,\mathbf{m}},2n(m-1)-((m-1)n+k))\\
        &-\HF(R/P_{2n,\mathbf{m}},2n(m-1)-((m-1)n+k+1))\\
        =&
        \HF(R/P_{2n,\mathbf{m}},(m-1)n-k)-\HF(R/P_{2n,\mathbf{m}},(m-1)n-k-1)\\
        =& \HS(R/I_{2n,\mathbf{m},1},(m-1)n-k)\,,
    \end{align*}
    which yields the claim.
\end{proof}

Our description of the rows of the $(m-1)$-Catalan triangle enables us to establish their log-concavity.

\begin{corollary}\label{cor:Log_convexity_rows_m-1_Catalan}
    The rows of the $(m-1)$-Catalan triangle are log-concave.
\end{corollary}
\begin{proof}
  The proof of Proposition~\ref{prop:Combinatorial_interpretation_s_Catalan_triangle} shows that the $n$th row of the $(m{-}1)$-Catalan triangle is given by the nonzero entries of the sequence of first differences of the vector of coefficients of $\HS(R/P_{2n,\mathbf{m}};t)$. In fact, it contains all the nonzero entries of that sequence, because in the column indexed by $k = 0$, we take the coefficient of $t^{(m{-}1)n}$ in that Hilbert series. But $(m{-}1)n$ is the socle degree of the algebra $R/I_{2n,\mathbf{m},1}$, i.e., the last nonzero entry in the sequence of first differences. Indeed, using the notation of Lemma~\ref{lem:Non-empty_sets_Crit_n_and_their_degree_structure}, 
$\delta_{2n}=\left\lfloor (D_{2n}+k-1)/2 \right\rfloor=\left\lfloor (2n(m-1)+1-1)/2 \right\rfloor=n(m-1)$.
The log-concavity of this positive part of the sequence of first differences was established in~\cite{142657}. (Note that in~\cite{142657}, the author twice writes ``log-convex'' where ``log-concave'' is meant.)
\end{proof}

\subsection{Weak Lefschetz property in positive characteristic}\label{sub:WLP}

In Corollary~\ref{cor:has_SLP}, we provided a new proof that all monomial complete intersections have the SLP when the base field has characteristic zero. However, over fields of positive characteristic, such ideals may fail to have either the WLP or the SLP. While a complete characterization exists for the SLP in this setting \cite{LUNDQVIST2019213}, the WLP is less understood. In particular, for general mixed degrees, a full characterization of the WLP is still unknown. Partial results  are given in 
\cite{brennerkaid, COOK2, HMNW,Kustin,KMRR, LUNDQVIST2019213}.

Nevertheless, in the equigenerated case, the WLP is classified, and for the case  $n \geq 5$ we have 
\begin{equation}\label{eq:wlp}
R/P_{n,\mathbf{m}} \text{ has the WLP in characteristic } p \text{ if and only if }  p > \left\lfloor \frac{n(m-1) +1}{2} \right\rfloor.
\end{equation}

This result is due to 
Kustin and Vraciu \cite{Kustin} for $\mathbf{k}$ infinite and  the second author and Nicklasson \cite{LUNDQVIST2019213} for $\mathbf{k}$ finite. We will give an alternative condition of the WLP in characteristic $p$ when $n \geq 5.$

Before stating this result, we first establish the following folklore result, for which we were unable to locate a suitable reference.

\begin{proposition}\label{prop:GB_mod_p}
Let $I = (f_1, \dots, f_t) \subseteq \mathbb{Q}[x_1, \dots, x_n]$ be a homogeneous ideal, and let $G = \{g_1, \dots, g_s\}$ be a Gr\"obner basis for $I$, where the denominators have been cleared so that all coefficients of the $f_i$ and $g_i$ are integers. If the coefficients of all leading terms $\ini(g_i)$ are not divisible by a prime $p$, then the reduction of $G$ modulo $p$ is a Gr\"obner basis for the reduction of $I$ in $\mathbb{F}_p[x_1, \dots, x_n]$.
\end{proposition}

\begin{proof}
Assume that the leading coefficients of all $g_i$ are not divisible by $p$. Then the S-polynomials $S(g_i, g_j)$ are constructed in the same way over $\mathbb{F}_p$ as over $\mathbb{Q}$. We know that the S-polynomials $S(g_i, g_j)$ reduce to zero modulo $G$ over $\mathbb{Q}$. We must show that the same holds over $\mathbb{F}_p$. Assume not, and let $d$ be the smallest degree of some $S(g_i, g_j)$ that does not reduce to zero over $\mathbb{F}_p$. This implies that 
\[
\dim_{\mathbb{Q}}(\mathbb{Q}[x_1, \dots, x_n]/I)_d > \dim_{\mathbb{F}_p}(\mathbb{F}_p[x_1, \dots, x_n]/I)_d,
\]
since the non-reduced S-polynomial gives an additional generator in degree $d$ over $\mathbb{F}_p$. This is in turn equivalent to
\begin{equation}\label{eq:bad_deg_ineq}
\dim_{\mathbb{Q}}(I)_d < \dim_{\mathbb{F}_p}(I)_d.
\end{equation}
Now, suppose each $f_i$ has degree $d_i$. Consider the matrix $M$ whose columns are indexed by all monomials of degree $d$, and whose rows are indexed by polynomials of the form $f_i m_j$, where $i = 1, \ldots, t$ and $m_j$ ranges over all monomials of degree $d - d_i$. In the row corresponding to $f_i m_j$, the entry in the column indexed by the monomial $t$ is the coefficient of $t$ in $f_i m_j$.

The dimension $\dim_{\mathbb{Q}}(I)_d$ equals the rank of $M$, i.e., the size of the largest nonzero minor of $M$ over $\mathbb{Q}$. Since all entries of $M$ are integers, $\dim_{\mathbb{F}_p}(I)_d$ can similarly be computed as the size of the largest nonzero minor of $M$ over $\mathbb{F}_p$. If a minor is nonzero over $\mathbb{F}_p$, it is also nonzero over $\mathbb{Q}$, implying $\dim_{\mathbb{Q}}(I)_d \geq \dim_{\mathbb{F}_p}(I)_d$, contradicting~\eqref{eq:bad_deg_ineq}. Hence, the S-polynomials reduce to zero modulo $G$ over $\mathbb{F}_p$, and the proof is complete.
\end{proof}

\begin{theorem}\label{thm:Wlp_cond}
Let $\mathbf{k}$ be a field of characteristic $p$, and fix integers $m \geq 2$ and $n \geq 5$. Then $\mathbf{k}[x_1, \dots, x_n]/(x_1^m, \dots, x_n^m) = R/P_{n,\mathbf{m}}$ has the WLP if and only if $I_{n,\mathbf{m},1}$ has the same initial ideal as it does when $\mathbf{k}$ is replaced by a field of characteristic zero.
\end{theorem}

\begin{proof} 
By \eqref{eq:wlp}, $R/P_{n,\mathbf{m}}$ has the WLP if and only if $p > \left\lfloor \frac{n(m-1) + 1}{2} \right\rfloor$. Moreover, applying the Buchberger's criterion, the initial ideal of $I_{n,\mathbf{m},1}$, viewed as an ideal in $\mathbf{k}[x_1,\ldots,x_n]$, agrees with the initial ideal of $I_{n,\mathbf{m},1}$ when viewed as an ideal in $\mathbf{k}'[x_1,\ldots,x_n]$, where $\mathbf{k}'$ is an extension field of $\mathbf{k}$. Hence, the structure of the initial ideal and the presence of the WLP depend only on the characteristic of $\mathbf{k}$. In particular, we may assume $\mathbf{k} = \mathbb{F}_p$.
Now, let $g$ be an element of the Gr\"obner basis of $I_{n,\mathbf{m},1}$ as described in Theorem~\ref{thm:Reduced_GB}. The largest prime dividing any coefficient of $g$ is at most $\deg(g)$. By Remark~\ref{rem:largest_degree_of_GB_element}, the maximal degree attained by such $g$ is  
$\left\lceil \frac{n(m-1) + 1 - 1}{2} \right\rceil = \left\lceil \frac{n(m-1)}{2} \right\rceil = \left\lfloor \frac{n(m-1) + 1}{2} \right\rfloor$.
Therefore, by Proposition~\ref{prop:GB_mod_p}, the initial ideal of $R/I_{n,\mathbf{m},1}$ remains unchanged from the characteristic zero case whenever $p$ exceeds this bound. Conversely, if $p$ is less than or equal to this number, then $R/P_{n,\mathbf{m}}$ fails the WLP.

Since the WLP of $R/P_{n,\mathbf{m}}$ is determined by the Hilbert series of $R/I_{n,\mathbf{m},1}$, and $R/P_{n,\mathbf{m}}$ has the WLP in characteristic zero, it follows that the Hilbert series of $R/I_{n,\mathbf{m},1}$ must differ from the characteristic zero case for $p \leq \left\lfloor \frac{n(m-1) + 1}{2} \right\rfloor$. Consequently, the initial ideal of $I_{n,\mathbf{m},1}$ also differs in this range.

In summary, $R/P_{n,\mathbf{m}}$ has the WLP and $I_{n,\mathbf{m},1}$ shares the characteristic zero initial ideal if and only if $p > \left\lfloor \frac{n(m-1) + 1}{2} \right\rfloor$, whereas $R/P_{n,\mathbf{m}}$ fails the WLP and $I_{n,\mathbf{m},1}$ has a different initial ideal compared to the characteristic zero case if and only if $p \leq \left\lfloor \frac{n(m-1) + 1}{2} \right\rfloor$.
\end{proof}
The technique from Theorem~\ref{thm:Wlp_cond} also applies to the mixed degree case, as follows.
\begin{proposition}
    Let $G_{n,\mathbf{m},1}$ be the Gr\"obner basis of $I_{n,\mathbf{m},1}$ as described in Theorem 
    \ref{thm:Reduced_GB}, but  
    where the denominators have been cleared so that all coefficients are integers. Let $p$ be a prime and let $\mathbf{k}$ be a field of characteristic $p$. If none of the coefficients of the leading terms of $G_{n,\mathbf{m},1}$ are divisible by $p$, then 
    $\mathbf{k}[x_1,\ldots,x_n]/P_{n,\mathbf{m}}$ has the WLP.
\end{proposition}

However, note that an if and only if statement of the form as in Theorem~\ref{thm:Wlp_cond} is no longer true when mixed degrees are allowed.
\begin{example}
Consider the family of algebras $B_a = \mathbb{F}_3[x_1, \dots, x_5]/(x_1^2, x_2^2, x_3^2, x_4^{1+3a}, x_5^{2+3a})$ parametrized by $a \geq 0$. When $a = 0$, we have $B_0 \cong R/P_{4,2}$, which has the WLP by 
\cite[Theorem 5.1]{Kustin}. Then 
\cite[Theorem 4.8]{LUNDQVIST2019213} implies that $B_a$ has the WLP for all $a \geq 0$. 
However, the initial ideal $\ini(I_{5,\mathbf{m}_a,1})$, where $\mathbf{m}_a = (2,2,2,1+3a,2+3a)$, differs over $\mathbb{F}_3$ compared to over $\mathbb{Q}$ when $a \geq 1$. Indeed, for all $a \geq 1$, the critical monomial $s = x_3x_4^2$ corresponds to a Gr\"obner basis element (after clearing denominators) given by
\[
g_s = 3x_3(x_4+x_5)^2 + (x_4+x_5)^3.
\]
Over $\mathbb{Q}$, $g_s$ has $x_3x_4^2$ as its leading monomial, but over $\mathbb{F}_3$, it appears that $x_4^3$ becomes the leading monomial instead, which is not a leading monomial over $\mathbb{Q}$. One can verify that $x_4^3$ is a leading monomial over $\mathbb{F}_3$ by noting that 
\[
x_4^3 + x_5^3 = (x_1 + x_2 + x_3 + x_4 + x_5)^3 \in I_{5,\mathbf{m}_a,1},
\]
thus showing that the initial ideals differ even though the algebra has the WLP in both cases.
\end{example}

\noindent{\bf Acknowledgments.}
Experiments with the computer algebra system Macaulay2~\cite{M2} were of great use during the preparation of this paper. The authors thank Eric Dannetun for pointing out the reference \cite{142657} to us. S.L. was supported by the Swedish Research Council grant VR2022-04009. 
M.O. and F.M. were partially supported by the FWO grants G0F5921N (Odysseus) and G023721N, and by the KU Leuven grant iBOF/23/064, and they acknowledge the hospitality of the Mathematics
Department at Stockholm University, where part of this work was carried
out. F.M. also gratefully acknowledges the travel support by the Wenner-Gren Foundation.

\bibliographystyle{abbrv}
\bibliography{Motzkin.bib}

\end{document}